\documentclass[11pt]{article}
\usepackage{amsmath}
\usepackage{amssymb}
\usepackage{amsbsy}
\usepackage{amsthm}
\usepackage{epsfig}
\usepackage{wrapfig}
\usepackage{eepic}
\usepackage{color}
\usepackage{graphicx}
\usepackage{amsfonts}%
\usepackage{float}
\usepackage{subfig}
\usepackage[toc,page]{appendix}  % Appendix package
\usepackage{xcolor}

\newtheorem{theorem}{Theorem}[section]
\newtheorem{Corollary}{Corollary}[section]
\newtheorem{remark}{Remark}[section]
\newcommand{\bu}{{u}}

\newcommand{\bD}{{D}}

\def\div{\operatorname{div}}

\newcommand{\blf} {f}

\newcommand{\divergence}{\nabla \cdot}
\newcommand{\norm}[1]{\left|\left|{}#1\right|\right|}
\newcommand{\RA}{\Rightarrow}

\newcommand{\inv}{^{-1}}

\newcommand{\lip}[2]{\left({}#1,#2\right){}}
\newcommand{\brak}[1]{\left[{}#1\right]}

\newcommand{\pare}[1]{\left({}#1\right)}
\newcommand{\linfnorm}[1]{\left|\left|{}#1\right|\right|_{L^{\infty}}}
\newcommand{\omint}{\int_{\Omega}}
\newcommand{\bmat}[1]{\begin{bmatrix} #1 \end{bmatrix}}
\newcommand{\curly}[1]{\left\{{}#1\right\}}

\begin{document}

\title{An energy, momentum and angular momentum conserving scheme for a regularization model of incompressible flow}

\author{
Sean Ingimarson
\footnote{School of Mathematical Sciences and Statistics, Clemson University, Clemson, SC 29634 (singima@clemson.edu);partially supported by NSF grant DMS 2011490.}
%\and
%Maxim A. Olshanskii
%\footnote{Houston}
%\and
%Leo G. Rebholz
%\footnote{Department of Mathematical Sciences, Clemson University, Clemson, SC 29634 (rebholz@clemson.edu); partially supported by NSF grant DMS 2011490}
}
\date{}

\maketitle

\begin{abstract}
We introduce a new regularization model for incompressible fluid flow, which is a regularization of the EMAC \textcolor{blue}{(energy, momentum, and angular momentum conserving)} formulation of the Navier-Stokes equations (NSE) that we call EMAC-Reg.  The EMAC formulation has proved to be a useful formulation because it conserves energy, momentum and angular momentum even when the divergence constraint is only weakly enforced.  However, it is still a NSE formulation and so cannot resolve higher Reynolds number flows without very fine meshes.  By carefully introducing regularization into the EMAC formulation, we create a model more suitable for coarser mesh computations but that still conserves the same quantities as EMAC, i.e., energy, momentum, and angular momentum.  We show that EMAC-Reg, \textcolor{blue}{when semi-discretized with a finite element spatial discretization is well-posed and optimally accurate.}    Numerical results are provided that show EMAC-Reg is a robust coarse mesh model. 
\end{abstract}

\section{Introduction}

The evolution of incompressible, viscous, Newtonian flow is determined by the Navier-Stokes equations (NSE), which are given by
\begin{eqnarray}
\bu_t + (\bu\cdot\nabla) \bu + \nabla p - \nu\Delta \bu & = & \blf, \label{nse1} \\
\div\bu & = & 0, \label{nse2}
\end{eqnarray}
with $\bu$ and $p$ representing velocity and pressure, $\blf$ an external force, $\nu$ the kinematic viscosity which is inversely proportional to the Reynolds number $Re$, and with appropriate boundary and initial conditions.  When the Reynolds number gets large, simulation of \eqref{nse1}-\eqref{nse2} can be very computationally expensive or even intractable due to the need for extremely fine spatial meshes in order to capture the smallest active scales of the flow.  As Kolmogorov showed in 1941, the smallest active scales in a flow are $O(Re^{-3/4})$ \cite{K41_1,K41_2,K41_3,F95}, which means that fully resolving 3D flows requires mesh widths to be $\Delta x = \Delta y = \Delta z = O(Re^{-3/4})$, and thus the total number of meshpoints to be $O(Re^{9/4})$.  \textcolor{blue}{Since industry routinely needs to simulate flows with $Re=10^6$ and larger (e.g. $Re\approx 10^6$ for a compact car at 60 mph \cite{laytonbook}), it becomes evident that such simulations can be very costly. In some cases, it is not possible since needing to solve the resulting linear systems at each nonlinear iteration in each time step requires both vast computational resources and the need to wait for weeks or months for a simulation to finish.}   Not resolving all the active scales in a NSE simulation is well known to create inaccuracy and numerical instability \cite{BIL06}.

To address this problem, many models have been developed that approximate the NSE, but can be solved on much coarser meshes than the NSE requires.  Some common examples are $k-\epsilon$ type models \cite{P00}, Smagorinsky type models \cite{Sma63,BIL06}, and our interest herein is a model that fits in the class of Large Eddy Simulation (LES) models \cite{BIL06,LR12,RL14,DE06}.  \textcolor{blue}{LES models aim to accurately estimate the large scales of the flow} and model the effect of small scales  on the large scales.  Hence the goal of LES is not pointwise accurate  solutions, but instead solutions that agree with NSE on general or averaged flow behavior.  Regularization models are our particular interest herein, and these are LES models that use filtering/averaging operations applied to the NSE to reduce the complexity of the system by eliminating finer scales and steepening the slope of the energy cascade at scales smaller than can be resolved on a given mesh.  \textcolor{blue}{NS-$\alpha$ and Leray-$\alpha$ (and their many variants), for example, are popular regularization models that have been extensively studied in recent years.  They are found to have many desirable mathematical properties} (e.g. well-posedness, fidelity to various physical balances \cite{FHT01,FHT02,R07b,CHOT05,GH03,LR12,C10,XWWI18}) and be successful in simulating high Reynolds number and turbulent flows \cite{HN01,CHOT05,GH06,CFHOTY98,CFOTW99a,RKB17,CDMR14,BKNRF13}.

Despite the attractive analytical properties that regularization models enjoy, many of these properties can be lost in a discretization, in particular the balances of physical quantites.  In fact, this is also true for discretizations of the NSE, which after being discretized generally loses the exact balances of most (if not all) of energy, momentum, angular momentum, helicity, and others.   The situation for discretized regularization models is worse, since the modeling process already removes or alters some important physical balances, and then the discretization process exacerbates the problem.  Of course, in any model of the NSE one must sacrifice some physical accuracy, since one is no longer solving the NSE.  However, the goal for simulations should be to maintain as much physical accuracy as possible, \textcolor{blue}{so that solutions are physically meaningful and can be confidently used as surrogates for the true physical model.}

To address the issue of discretizations not preserving important physical balances, the EMAC scheme was developed for the NSE in \cite{CHOR17} and by design it conserves
energy, momentum, and angular momentum even when the divergence constraint is not strongly enforced (which is the typical case with finite element and finite difference methods).  It is the first such scheme to conserve these quantities in typical finite element discretizations where the divergence constraint is weakly enforced.   We note that if newly developed pointwise divergence-free finite elements are used, e.g. \cite{arnold:qin:scott:vogelius:2D,Z05,evans2013isogeometric_s,GN14b,GN14}, then the numerical velocity found with EMAC will be the same as recovered from more traditional convective and skew-symmetric formulations, and all of them w\textcolor{blue}{ill conserve energy, momentum, and angular momentum.  However, the development of these strongly divergence-free methods is still quite new.  It typically requires non-standard meshing and elements} and is not yet included into major finite element software packages such as deal.ii \cite{dealII_92}.

In the series of works that followed the original EMAC paper, e.g. \cite{CHOR19,PCLRH18,OCAM19,LPH19,MSLGD19,SP18,SP18b}, the EMAC scheme was shown to be highly accurate especially over longer time intervals compared to standard schemes.  Theory to justify this longer time accuracy was provided in \cite{OR20}, \textcolor{blue}{where it was shown that using EMAC dramatically reduces the usual upper bounds on the error through removal of the Reynolds number dependence in the Gronwall exponent. Moreover, it also eliminates lower bounds on the error caused by standard schemes' lack of momentum conservation.}  It is not surprising that EMAC was able to achieve better accuracy, given the long history of enhanced physics schemes performing better than standard schemes, especially over long time intervals: this idea began with Arakawa's 2D energy and enstrophy conserving scheme in 1966 \cite{A66}, and then continued on to many other physical system discretizations, e.g. \cite{ST89,AM03,LW04,R07,evans2013isogeometric,SCN15,PG16,F75}. 

The purpose of this paper is to extend the EMAC technology to a regularization model for the NSE.  As discussed above, \textcolor{blue}{for larger $Re$ it becomes necessary to replace the NSE with a surrogate model.  In practice,} it is desirable for the surrogate model and its discretized solution to have as high a degree of physical accuracy as possible.  The regularization model that we will discretize takes the form
\begin{eqnarray}
\bu_t + w\cdot\nabla w + \nabla p - \nu\Delta \bu & = & \blf, \label{reg1} \\
-\alpha^2\Delta w + w & = & u, \label{reg2a}\\
\div w & = & 0. \label{reg2}
\end{eqnarray}
Here,  \eqref{reg2a} denotes the Helmholtz filter with filtering radius $\alpha$, and we note this system could be written as a fourth order problem in $w$.  Under periodic boundary conditions, $\div w=0 \implies \div u=0$ since the filter operation commutes with the divergence operator.  Also for smooth solutions to \eqref{reg1}-\eqref{reg2}, we may apply $\div$ to \eqref{reg2a} and recover $\div u=0$.  For weak solutions that are not strong, if one wishes to enforce additionally that $\div u=0$, then one may add a Lagrange multiplier (i.e. pressure) term to \eqref{reg2a}; however, we assume throughout that strong solutions to \eqref{reg1}-\eqref{reg2} exist under the assumption of homogeneous Dirichlet boundary conditions for $u$ and $w$.  In the periodic case, the mathematical theory for the well-known NS-$\alpha$ model can be extended to this model in a straightforward way from \cite{FHT01,FHT02}, providing well-posedness as well as regularity of solutions depending on regularity of the problem data.

While \eqref{reg2} denotes the Helmholtz filter with filtering radius $\alpha$, other types of regularization operators are possible such other filters or approximately deconvolved filters \cite{LR12,BIL06,CL14,sto022,ada021,AS99}.  In this initial study, we consider only the Helmholtz filter.

The specific formulation for the regularization model \eqref{reg1}-\eqref{reg2} is specifically chosen so that it fits the form of a model that the nonlinearity formulation $2D(w)w + (\div w)w$ (which is identical to $w\cdot\nabla w$ when $\div w=0$) will preserve energy, momentum and angular momentum when $\div w \neq 0$.  We note that the nonlinear forms $w\cdot \nabla u$ (i.e. that of Leray-$\alpha$) or $(\nabla \times u)\times w$ (i.e. that of NS-$\alpha$) do not preserve each of energy, momentum and angular momentum when pointwise divergence free is lost for velocities and/or filtered velocities (see appendix A).  In other words, \eqref{reg1}-\eqref{reg2} is the $\alpha$ regularization model that naturally fits the EMAC framework.

Herein, we study and test discretizations of \eqref{reg1}-\eqref{reg2} as well as show that it conserves energy, momentum, and angular momentum; we denote it as the Energy, Momentum, and Angular Momentum conserving regularization formulation (EMAC-Reg).  We formally define this scheme in Section 3, followed by showing it is stable, well-posed, the aforementioned quantities are conserved, and is optimally accurate.

This paper is arranged as follows.  Section 2 is dedicated to introducing notation and preliminary information required for the analysis later in the paper.  Section 3 covers the EMAC-Reg scheme's stability, well-posedness, quantity conservation, and error analysis.  Section 4 describes several numerical experiments which tests the conservation properties and robustness over coarse meshes of EMAC-Reg, and compares it to other related models.

\section{Notation and preliminaries}

We consider a convex polygonal domain $\Omega \subset \mathbb{R}^d$, $d=2,3$.  Denote the $L^2(\Omega)$ inner product and norm on $\Omega$ by $(\cdot,\cdot)$ and $\| \cdot \|$, respectively, and we note all other norms will be labeled with subscripts.  The natural velocity and pressure spaces are
\begin{align*}
X=\curly{v \in H^1(\Omega)^d , \: v|_{\partial \Omega}=0}, \quad Q=\curly{q \in L^2(\Omega), \: \omint q dx=0}.
\end{align*}
Let $V$ denote the divergence-free subspace of $X$,
$
V:= \{ v \in X\,:\,\lip{\divergence v}{q}=0, \: \forall q \in Q\}.
$
\textcolor{blue}{We also define the dual of $X$ and its norm,}
\begin{align*}
\textcolor{blue}{X'=H\inv(\Omega), \quad \norm{f}_{X'}=\sup_{v \in X'}\frac{\lip{f}{v}}{\norm{v}_{X}},}
\end{align*}
\textcolor{blue}{for any $f \in L^2(\Omega)$.} 

We further consider subspaces $X_h \subset X$, $Q_h \subset Q$ to be finite element velocity and pressure spaces corresponding to a conforming triangulation $\mathcal{T}_h$ of $\Omega$, where $h$ is the global mesh-size. For $\mathcal{T}_h$ we assume the minimal angle condition if $h$ varies.  Define the discretely divergence-free subspace of $X_h$ by
\[
V_h:= \{ v_h \in X_h,\ (\nabla \cdot v_h,q_h)=0\ \forall q_h\in Q_h \}.
\]
Most common finite element (FE) discretizations of the NSE and related systems enforce the divergence constraint $\div \bu_h=0$ weakly and thus $V_h \not\subset V$.

Define the operator $A_h$ by:  Given $\phi \in H^1(\Omega)$, $A_h\phi \in V_h$ solves
\begin{align}
\lip{A_h\phi}{v_h}=-\lip{\nabla \phi}{\nabla v_h}
\label{eq:StokesProj}
\end{align}
for all $v_h \in V_h$.

We denote $I_{St}^h$ as the \textcolor{blue}{discrete Stokes projection operator \cite{SL17}, which is defined by}: Given $\phi\in V$, find $I_h^{St}\phi \in V_h$ satisfying
\begin{align}
(\nabla I_h^{St}\phi,\nabla v_h) = (\nabla \phi,\nabla v_h) \quad \forall v_h \in V_h.
\label{eq:StokesOp}
\end{align}
Following \cite{SL17}, \textcolor{blue}{we assume $X_h=X\cap P_k(\mathcal{T}_h)$ and $Q_h=Q\cap P_{k-1}(\mathcal{T}_h)$, with $P_k$ being the set of degree $k$ polynomials.  Additionally we assume that these spaces satisfy the infsup condition and} the domain is such that the discrete Stokes operator satisfies for all divergence-free $\phi \in H^{k+1}(\Omega)$
\begin{align}
&\norm{\nabla \pare{\phi-I_h^{St}(\phi)}}\leq Ch^k|\phi|_{k+1},\label{eq:etabound}\\
&\norm{\phi-I_h^{St}(\phi)}\leq Ch^{k+1}|\phi|_{k+1}.\label{eq:etaboundgrad}
\end{align}

\subsection{Vector identities and trilinear forms}
\textcolor{blue}{The EMAC formulation uses the identity}
\begin{equation}
(\bu\cdot \nabla) \bu = 2\bD(\bu)\bu - \frac12\nabla |\bu|^2, \label{vecid7}
\end{equation}
\textcolor{blue}{where $u$ is a sufficiently smooth gradient field and $D(u)=\frac{1}{2}(\nabla u+(\nabla u)^T)$ is the symmetric part of the gradient.  This identity splits the inertia term into the acceleration} driven by $2\bD(\bu)$ and \textcolor{blue}{potential term that is absorbed by redefined pressure (defined in Theorem 3.4).}
Based on \eqref{vecid7} the trilinear form for EMAC (Galerkin) formulation is defined by
\[
c(u,v,w) =2(D(u)v,w)+(\mbox{div}(u)v,w).
\]
\textcolor{blue}{Herein, we assume $u,\ v,\ w\in X$ (no divergence free condition is assumed for any of $u,\:v,\:w$).}  
\textcolor{blue}{The trilinear form $c$ was developed for EMAC in \cite{CHOR17} as a consistent weak representation of the NSE nonlinearity that preserves energy, momentum, and angular momentum.  Because we apply this treatment to the filtered equations \eqref{reg1}-\eqref{reg2}, we must redefine the pressure to}
\begin{align*}
\textcolor{blue}{P=p-\frac{1}{2}|w|^2.}
\end{align*}
\textcolor{blue}{We will also be utilizing these next identities several times throughout the paper,}
\begin{align}
&\lip{u \cdot \nabla v}{w}=-\lip{u \cdot \nabla w}{v}-\lip{(\divergence u)v}{w}, \label{eq:form1}\\
&\lip{u \cdot \nabla w}{w}=-\frac{1}{2}\lip{(\divergence u)w}{w}, \label{eq:form2}\\
&\lip{u \cdot \nabla v}{w}=\lip{(\nabla v)u}{w}=\lip{(\nabla v)^Tw}{u}. \label{eq:form3}
\end{align}

\section{The EMAC-Reg scheme and its analysis}

We consider the following semi-discretization of the EMAC-Regularization: Given $w_h^0,u_h^0\in V_h$, find \textcolor{blue}{$(u_h,P_h,w_h,\lambda_h)$} $\in$ $(X_h,Q_h,X_h,Q_h)\times(0,T]$ for all $(v_h,q_h,\chi_h,r_h)$\\ $\in (X_h,Q_h,X_h,Q_h)$,
\begin{align}
\lip{(u_h)_t}{v_h}+c(w_h,w_h,v_h)-\lip{P_h}{\divergence v_h}+\nu\lip{\nabla u_h}{\nabla v_h}&=\lip{f}{v_h}, \label{eq:emacreg1}\\
\lip{\divergence u_h}{q_h}&=0,\label{eq:emacreg2}\\
\lip{\lambda_h}{\divergence \chi_h}+\alpha^2\lip{\nabla w_h}{\nabla \chi_h}+\lip{w_h}{\chi_h}&=\lip{u_h}{\chi_h},\label{eq:emacreg3}\\
\lip{\divergence w_h}{r_h}&=0.\label{eq:emacreg4}
\end{align}
Even though in the continuous model \eqref{reg1}-\eqref{reg2} the constraint $\div u=0$ follows from the model definition, in a discretization it does not.  Hence we enforce that the velocity be discretely divergence-free with \eqref{eq:emacreg2}, and thus \textcolor{blue}{add the Lagrange multiplier} to \eqref{eq:emacreg3} so the system is not overdetermined.

\textcolor{blue}{The analysis below is for the semi-discrete system \eqref{eq:emacreg1}-\eqref{eq:emacreg4}.  Extensions to particular timestepping methods can be done, and for A-stable methods such as backward Euler, BDF2, or Crank-Nicolson, analogous stability results can be obtained (with expected additional error depending on time step size).  For conservation properties, only Crank-Nicolson will allow for exact conservation; other timestepping methods' discrete time derivative term will generally contribute to the energy, momentum, and angular momentum balances.  It is believed, with some evidence from \cite{CHOR17,CHOR19} that this effect is much less problematic than the nonlinear term contributing to these balances.  While it is discussed in \cite{CHOT05} that Crank-Nicolson will provide optimal convergence in space and time, whether an improved Gronwall constant can be obtained is an open problem.}

\subsection{Stability and well-posedness}

We first prove stability of the scheme, followed by well-posedness.
\begin{theorem}
Let $(u_h,P_h,w_h,\lambda_h)$ solve \eqref{eq:emacreg1}-\eqref{eq:emacreg4} with \textcolor{blue}{$f\in L^2(0,T;X')$} and $w_h^0 \in H^1(\Omega)$.  The following bounds hold:
\begin{align}
\norm{w_h}^2&+\alpha^2\norm{\nabla w_h}^2+\nu\brak{\norm{\nabla w_h}^2_{L^2(0,T;L^2)}+\alpha^2\norm{A_hw_h}^2_{L^2(0,T;L^2)}}\notag\\
&\leq \nu\inv\textcolor{blue}{\norm{f}^2_{L^2(0,T;X')}}+\norm{w_h(0)}^2+\alpha^2\norm{\nabla w_h(0)}^2, \label{eq:3.1bound}\\
\norm{u_h}^2_{L^2(0,T;L^2)}&\leq (T+\alpha^2)\nu\inv\textcolor{blue}{\norm{f}^2_{L^2(0,T;X')}}+\norm{w_h(0)}^2\notag\\
&\hspace{28mm}+\alpha^2\norm{\nabla w_h(0)}^2+\textcolor{blue}{\norm{u_h(0)}^2}.\notag
\end{align}
\end{theorem}
\begin{proof}
Choosing $v_h=w_h$ in \eqref{eq:emacreg1} gives us
\begin{align}
\lip{(u_h)_t}{w_h}+\nu\lip{\nabla u_h}{\nabla w_h}=\lip{f}{w_h},\label{eq:og}
\end{align}
where the trilinear term $c(w_h,w_h,w_h)=0$ and $\textcolor{blue}{\lip{P_h}{\divergence w_h}}=0$ by \eqref{eq:emacreg4}.  

Next choose $\chi_h=A_hw_h$ in \eqref{eq:emacreg3} and then $\chi_h=w_h$ in \eqref{eq:emacreg3} to obtain
\begin{align}
\alpha^2\norm{A_hw_h}^2+\norm{\nabla w_h}^2=\lip{\nabla u_h}{\nabla w_h}, \label{eq:2}\\
\alpha^2 \norm{\nabla w_h}^2+\norm{w_h}^2=\lip{u_h}{w_h}.\label{eq:1}
\end{align}
To find another expression for $\lip{(u_h)_t}{w_h}$, we take the temporal derivative of \eqref{eq:1} to get
\begin{align*}
&\alpha^2\frac{d}{dt}\norm{\nabla w_h}^2+\frac{d}{dt}\norm{w_h}^2=\frac{d}{dt}\lip{u_h}{w_h}=\lip{(u_h)_t}{w_h}+\lip{u_h}{(w_h)_t},
\end{align*}
and so
\begin{align}
\lip{(u_h)_t}{w_h}=\frac{d}{dt}\brak{\norm{w_h}^2+\alpha^2\norm{\nabla w_h}^2}-\lip{u_h}{(w_h)_t}. \label{eq:timeterm1}
\end{align}
Choosing $\chi_h =(w_h)_t$ in \eqref{eq:emacreg3} yields
\begin{align}
\frac{\alpha^2}{2}\frac{d}{dt}\norm{\nabla w}^2+\frac{1}{2}\frac{d}{dt}\norm{w_h}^2=\lip{u_h}{(w_h)_t}, \label{eq:timeterm2}
\end{align}
and now combining \eqref{eq:timeterm1} and \eqref{eq:timeterm2} provides
\begin{align}
\lip{(u_h)_t}{w_h}&=\frac{d}{dt}\brak{\norm{w_h}^2+\alpha^2\norm{\nabla w_h}^2}-\frac{d}{dt}\brak{\frac{\alpha^2}{2}\norm{\nabla w_h}^2+\frac{1}{2}\norm{w_h}^2}\notag\\
&=\frac{1}{2}\frac{d}{dt}\brak{\norm{w_h}^2+\alpha^2\norm{\nabla w_h}^2}. \label{eq:timeterm}
\end{align}
Using \eqref{eq:2} and \eqref{eq:timeterm} in \eqref{eq:og} gives us 
\begin{align}
\frac{1}{2}\frac{d}{dt}\brak{\norm{w_h}^2+\alpha^2\norm{\nabla w_h}^2}+\nu\brak{\norm{\nabla w_h^2}+\alpha^2\norm{A_hw_h}^2}=\lip{f}{w_h}. \label{eq:gang}
\end{align}
We bound the right hand side of \eqref{eq:gang} using Cauchy-Schwarz and Young's inequalities,
\begin{align*}
\lip{f}{w_h}\leq \textcolor{blue}{\norm{f}_{X'}}\norm{w_h}\leq \frac{\nu}{2}\norm{\nabla w_h}^2+\frac{1}{2\nu}\textcolor{blue}{\norm{f}^2_{X'}},
\end{align*}
and after inserting this bound into \eqref{eq:gang} and then integrating in time gives us the first stated estimate
\begin{align*}
&\norm{w_h}^2+\alpha^2\norm{\nabla w_h}^2+\nu\brak{\norm{\nabla w_h}^2_{L^2(0,T;L^2)}+\alpha^2\norm{A_hw_h}^2_{L^2(0,T;L^2)}}\\
\leq &\nu\inv\textcolor{blue}{\norm{f}^2_{L^2(0,T;X')}}+\norm{w_h(0)}^2+\alpha^2\norm{\nabla w_h(0)}^2.
\end{align*}

To get the second bound, we test \eqref{eq:emacreg3} with $\chi_h=u_h$ to give us
\begin{align}
\norm{u_h}^2=\alpha^2\lip{\nabla w_h}{\nabla u_h}+\lip{w_h}{u_h}. \label{eq:uheq1}
\end{align}
Applying \eqref{eq:2} and \eqref{eq:1}, \eqref{eq:uheq1} becomes
\begin{align}
\norm{u_h}^2=\alpha^2\brak{\alpha^2\norm{A_hw_h}^2+\norm{\nabla w_h}^2}+\alpha^2\norm{\nabla w_h}^2+\norm{w_h}^2. \label{eq:uheq2}
\end{align}
Now using \eqref{eq:3.1bound} in \eqref{eq:uheq2} and integrating with respect to time, we get
\begin{align*}
\norm{u_h}^2_{L^2(0,T;L^2)}\leq & \alpha^2\brak{\alpha^2\norm{A_hw_h}^2_{L^2(0,T;L^2)}+\norm{\nabla w_h}^2_{L^2(0,T;L^2)}}+T\nu\inv\textcolor{blue}{\norm{f}^2_{L^2(0,T;X')}}\notag\\
&+\norm{w_h(0)}^2+\alpha^2\norm{\nabla w_h(0)}^2+\norm{u_h(0)}^2.
\end{align*}
Finally using \eqref{eq:3.1bound} to bound the the quantity in brackets yields the second result of the theorem:
\begin{align*}
\norm{u_h}^2_{L^2(0,T;L^2)}\leq& (T+\alpha^2)\nu\inv\textcolor{blue}{\norm{f}^2_{L^2(0,T;X')}}+\norm{w_h(0)}^2\\
&+\alpha^2\norm{\nabla w_h(0)}^2+\norm{u_h(0)}^2.
\end{align*}
\end{proof}

\begin{theorem}
The semidiscrete EMAC-Reg system \eqref{eq:emacreg1}-\eqref{eq:emacreg4} is well-posed.
\end{theorem}
\begin{proof}
The nonlinearity of the system of ODEs defining $w_h$ and $u_h$ is \textcolor{blue}{quadratic, and thus locally Lipschitz}.  By classical ODE theory we may conclude that a solution exists and is unique so long \textcolor{blue}{as all possible solutions cannot blow up in finite time \cite{Layton08}, which is established in Theorem 3.1 and} due to the LBB assumption on the finite element space it is true for $P_h$ and $\lambda_h$.
\end{proof}

\subsection{Conservation of energy, momentum, and angular momentum}

For NSE, energy, momentum, and angular momentum are defined by
\begin{align*}
\textcolor{blue}{\text{Kinetic energy}}\quad &\textcolor{blue}{E=\frac{1}{2}\lip{u}{u}=\frac{1}{2}\omint |u|^2d\textcolor{blue}{x},} \\
\textcolor{blue}{\text{Linear momentum}}\quad &\textcolor{blue}{M=\omint u \:d\textcolor{blue}{x},}\\
\textcolor{blue}{\text{Angular momentum}}\quad &\textcolor{blue}{M_x=\omint u \times x \:dx.}
\end{align*}
Let $e_1=(1,0,0)^T, \: e_2=(0,1,0)^T, \: e_3=(0,0,1)^T$ and $\phi_i=x \times e_i$ for $i=1,2,3$.  The definition of momentum and angular momentum can now be equivalently written as:
\begin{align*}
M_i&=\omint u_i\: dx =\lip{u}{e_i},\\
(M_x)_i&=\omint (u \times x) \cdot e_i \: dx=\lip{u}{\phi_i}.
\end{align*}
Conservation laws for models of fluid flow are determined by the model \textcolor{blue}{itself, and in some sense this should be analogous to the true physical law.}  For EMAC-Reg, we prove below that the model energy, momentum, and angular momentum conserved quantities are given by
\begin{align*}
&\textcolor{blue}{E_{EMAC-Reg}:=\frac{1}{2}\omint u \cdot w \: dx=\frac{1}{2}\lip{u}{w}=\frac{1}{2}\norm{w}^2+\frac{\alpha^2}{2}\norm{\nabla w}^2,}\\
&\textcolor{blue}{M_{EMAC-Reg}:=\omint u \: dx=\lip{u}{e_i},}\\
&\textcolor{blue}{AM_{EMAC-Reg}:=\omint u \times x \: dx=\lip{u}{\phi_i}.}
\end{align*}

\textcolor{blue}{For regularization models, altered quantities are generally conserved on the continuous level due to the regularization's modification of the physical equation.  \textit{Our goal is that conserved quantities in our discretization matches that of the continuous model.}}

Following \cite{CHOR17}, for our conservation law study, we notice that practical boundary conditions will alter the balance of the quantities above.  Enforcing conditions strongly or weakly will also impact the balance for the numerical experiments.  It is for these reasons that we isolate the affect on the treatment of the nonlinearity of the quantities from the boundary conditions.  \textcolor{blue}{We assume for the numerical experiments in section 4 that the finite element solution $u_h,$ $P_h$ is supported in a compact subdomain $\hat{\Omega}\subsetneq \Omega$.}  Not coincidentally, the conserved energy is the same as is conserved by the related NS-$\alpha$ model and the ADM turbulence model of Stolz and Adams \cite{LR12}.

\begin{theorem}
The EMAC-Reg formulation conserves energy, momentum, and angular momentum for $\nu=0, \: f=0$.
\begin{align*}
&E_{EMAC-Reg}(t)=E_{EMAC-Reg}(0)\\
&M_{EMAC-Reg}(t)=M_{EMAC-Reg}(0)\\
&AM_{EMAC-Reg}(t)=AM_{EMAC-Reg}(0)
\end{align*}
for all $t$.  \textcolor{blue}{Furthermore, we assume that the finite element solutions $u_h$ and $p_h$ are supported in a compact subdomain $\hat{\Omega}\subsetneq\Omega$ so that there exists a strip $S=\Omega \backslash \hat{\Omega}$ along $\partial \Omega$ where $u_h$ vanishes.  This is also assumed for the forcing function $f$.} 
\label{Conservation_Thm}
\end{theorem}
\begin{proof}
We start by showing energy conservation.  We test \eqref{eq:emacreg1} with $v=w_h$,
\begin{align*}
\lip{(u_h)_t}{w_h}+c(w_h,w_h,w_h)+\nu\lip{\nabla u_h}{\nabla w_h}=\lip{f}{w_h}.
\end{align*}
Since $c(w_h,w_h,w_h)=0$ and we assumed $f=0$ and $\nu=0$, we are left with
\begin{align*}
\lip{(u_h)_t}{w_h}=0 \RA \frac{d}{dt}\lip{u_h}{w_h}=0,
\end{align*}
and hence energy is preserved.

For momentum conservation, we cannot simply test \eqref{eq:emacreg1} with $e_i$, because it is not in the finite element space $X_h$.  To remedy this, we recall that $u_h=0$ on the strip $S$.  Define $\chi(g)\in X_h$ for $g=e_i$ and $\phi_i$ such that $\chi(g)=g$ in $\hat{\Omega}$ and arbitrary on $S=\Omega \backslash \hat{\Omega}$ to satisfy the boundary conditions of $X_h$.

For momentum conservation we test \eqref{eq:emacreg1} with $\chi(e_i)$, where $\chi(e_i)$ is the restriction of $e_i$ on $\Omega$.  This gives
\begin{align*}
\lip{(u_h)_t}{e_i}+c(w_h,w_h,e_i)+\nu\lip{\nabla u_h}{\nabla e_i}=\lip{f}{e_i}.
\end{align*}
For the nonlinear term
\begin{align*}
c(w_h,w_h,e_i)&=\lip{(\nabla w_h)w_h}{e_i}+\lip{(\nabla w_h)^Tw_h}{e_i}+\lip{(\divergence w_h)w_h}{e_i}\\
&=\lip{w_h \cdot \nabla w_h}{e_i}+\lip{w_h}{(\nabla w_h)e_i}+\lip{(\divergence w_h)w_h}{e_i},
\end{align*}
where we get the above identity from \eqref{eq:form3}.  Then applying \eqref{eq:form1} and \eqref{eq:form2} to the first term and using that $\nabla e_i=0$ gives
\begin{align*}
c(w_h,w_h,e_i)&=-\lip{w_h\cdot \nabla e_i}{w_h}-\lip{(\divergence w_h)w_h}{e_i}+\lip{e_i\cdot \nabla w_h}{w_h}+\lip{(\divergence w_h)w_h}{e_i}\\
&=\lip{e_i\cdot \nabla w_h}{w_h}\\
&=-\frac{1}{2}\lip{(\divergence e_i)w_h}{w_h}=0.
\end{align*}
Combining this with $f=\nu=0$, the fact that $\divergence e_i=0$ and that $e_i$ doesn't depend on t, we end up with
\begin{align*}
\frac{d}{dt}\lip{u_h}{e_i}=0,
\end{align*}
which is precisely momentum conservation.

For angular momentum, we test \eqref{eq:emacreg1} with $\chi(\phi_i)$, where $\chi(\phi_i)$ is the restriction of $\phi_i$ on $\Omega$, which gives  
\begin{align*}
\lip{(u_h)_t}{\phi_i}+c(w_h,w_h,\phi_i)+\nu\lip{\nabla u_h}{\nabla \phi_i}=\lip{f}{\phi_i}.
\end{align*}
We begin by investigating $\nabla \phi_i$.  First let $x=\bmat{x_1&x_2&x_3}^T$, then $\phi_i$ for $i=1,2,3$ takes the form
\begin{align*}
\phi_1=\bmat{0\\x_3\\-x_2}, \quad \phi_2=\bmat{-x_3\\0\\x_1}, \quad \phi_3=\bmat{x_2\\-x_1\\0}.
\end{align*}
This immediately gives
\begin{align}
\nabla \phi_1=\bmat{0&0&0\\0&0&-1\\0&1&0}, \quad \nabla \phi_2=\bmat{0&0&1\\0&0&0\\-1&0&0}, \quad \nabla \phi_3=\bmat{0&-1&0\\1&0&0\\0&0&0}.\label{eq:phi}
\end{align}
With this in mind, we now investigate the nonlinear term and apply \eqref{eq:phi} along with $\divergence \phi_i=0$ to get
\begin{align*}
c(w_h,w_h,\phi_i)&=-\lip{w_h \cdot \nabla \phi_i}{w_h}-\lip{(\divergence w_h)w_h}{\phi}+\lip{\phi_i \cdot \nabla w_h}{w_h}+\lip{(\divergence w_h)w_h}{\phi_i}\\
&=-\lip{w_h\cdot\nabla \phi_i}{w_h}-\frac{1}{2}\lip{(\divergence \phi_i)w_h}{w_h}\\
&=0.
\end{align*}
Combining this with $f=\nu=0$ provides
\begin{align*}
\frac{d}{dt}\lip{w_h}{\phi_i}=0,
\end{align*}
which is angular momentum conservation.
\end{proof}

\subsection{Error analysis}

In this subsection, we provide the error analysis for EMAC-Reg.  The discretized formulation is subtracted from the weak formulation and we manipulate the nonlinear term to get an appropriate error bound.  Below is the result.

\begin{theorem}
Let $(u_h,P_h,w_h,\lambda_h)$ solve \eqref{eq:emacreg1}-\eqref{eq:emacreg4} and $(u,p,w)$ solve \eqref{reg1}-\eqref{reg2} with $w_t \in L^2(0,T;X')$, $w,\nabla w \in L^1(0,T;L^{\infty})$ \textcolor{blue}{and $P \in L^2(0,T;L^2)$.}  Denote $e^u(t)=u(t)-u_h(t)$, $e^w(t)=w(t)-w_h(t)$.  We get the following error estimates:
\begin{align*}
\begin{split}
&\alpha^2\norm{\nabla e^w}^2+\norm{e^w}^2+\nu\int_0^T\brak{2\alpha^2\norm{A_he^w}^2+\norm{\nabla e^w}^2}dt\\
\leq & \alpha^2\norm{\nabla \eta^w}^2+\norm{\eta^w}^2+\nu\brak{2\alpha^2\norm{A_h\eta^w}^2_{L^2(0,T;L^2)}+\norm{\nabla \eta^w}^2_{L^2(0,T;L^2)}}\\
&+K\left[ 4\nu\inv\left(\norm{\eta^w_t}^2_{L^2(0,T;X')}+\inf_{q_h \in L^2(0,T;Q_h)}\norm{P-q_h}^2_{L^2(0,T;L^2)}\right)\right],
\end{split}
\end{align*}
where
\begin{align*}
K=\exp\brak{2C\pare{\norm{w}_{L^1(0,T;L^{\infty})}+\norm{\nabla w}_{L^1(0,T;L^{\infty})}}+T\nu},
\end{align*}
\textcolor{blue}{and}
\begin{align*}
\textcolor{blue}{\norm{e^u}^2\leq} &\textcolor{blue}{\alpha^2\norm{\nabla \eta^w}^2+\norm{\eta^w}^2+\nu\brak{2\alpha^2\norm{A_h\eta^w}^2_{L^2(0,T;L^2)}+\norm{\nabla \eta^w}^2_{L^2(0,T;L^2)}}}\\
&\textcolor{blue}{+K\left[ 4\nu\inv\left(\norm{\eta^w_t}^2_{L^2(0,T;X')}+\inf_{q_h \in L^2(0,T;Q_h)}\norm{P-q_h}^2_{L^2(0,T;L^2)}\right)\right]+3\norm{\eta^u}^2.}
\end{align*}
\label{error_theorem}
\end{theorem}
\textcolor{blue}{\begin{Corollary}
Let $(u_h,P_h,w_h,\lambda_h)$ solve \eqref{eq:emacreg1}-\eqref{eq:emacreg4} and $(u,p,w)$ satisfy the same criteria as in Theorem \ref{error_theorem}.  Additionally suppose $u,w \in L^{\infty}(0,T;H^{k+1})$, $w_t \in L^2(0,T;H^{k+1})$, and $P \in L^2(0,T;H^{k})$.  Assuming $X_h=X\cap P_k(\mathcal{T}_k)$ and $Q_h=Q\cap P_{k-1}(\mathcal{T}_k)$, where $P_k$ denotes the set of degree $k$ polynomials, satisfies the infsup condition, we get from \eqref{eq:etabound} and \eqref{eq:etaboundgrad}
\begin{align*}
\begin{split}
&\alpha^2\norm{\nabla e^w}^2+\norm{e^w}^2+\nu\int_0^T\brak{2\alpha^2\norm{A_he^w}^2+\norm{\nabla e^w}^2}dt\\
\leq & C\left(\alpha^2h^{2k}|w|_{k+1}^2+h^{2k+2}|w|_{k+1}^2+\nu\brak{2\alpha^2h^{2k-1}\norm{w}^2_{L^2(0,T;H^{k+1})}+h^{2k}\norm{w}^2_{L^2(0,T;H^{k+1})}}\right. \\ 
&\left.+K\left[ 4\nu\inv\left(h^{2k+2}\norm{w_t}^2_{L^2(0,T;H^{k+1})}+h^{2k}\norm{P}^2_{L^2(0,T;H^k)}\right)\right]\right),
\end{split}
\end{align*}
and 
\begin{align*}
&\norm{e^u}^2\\
\leq & C\left(\alpha^2h^{2k}|w|_{k+1}^2+h^{2k+2}|w|_{k+1}^2+\nu\brak{2\alpha^2h^{2k-1}\norm{w}^2_{L^2(0,T;H^{k+1})}+h^{2k}\norm{w}^2_{L^2(0,T;H^{k+1})}}\right.\\
+&\left.K\left[ 4\nu\inv\left(h^{2k+2}\norm{w_t}^2_{L^2(0,T;H^{k+1})}+h^{2k}\norm{P}^2_{L^2(0,T;H^{k})}\right)\right]+3h^{2k+2}|u|^2_{L^2(0,T;H^{k+1})}\right),
\end{align*}
where $K$ is defined in Theorem \ref{error_theorem}.
\end{Corollary}}
\begin{proof}
\textcolor{blue}{The result follows by applying standard interpolation estimates to Theorem 3.4 with the selected discrete spaces.}
\end{proof}

\begin{remark}
\textcolor{blue}{Theorem \ref{error_theorem} provides an optimal $L^2(0,T;H^1)$ error estimate for $w$} if \textcolor{blue}{$\alpha=O(h)$ (which we assume is the case).  An} additional feature is that the Gronwall constant does not depend explicitly on the inverse of the viscosity (which is not known to be true for discretizations of related models such as NS-$\alpha$ or Leray-$\alpha$ \cite{LR12}).
\end{remark}

\begin{proof}
\textcolor{blue}{We will break up the proof for Theorem \ref{error_theorem} into steps.}
\\
\\
\textcolor{blue}{\textbf{Step 1:}  Develop error equation.}
\\
\\
We begin by testing \eqref{reg1} by $v_h\in V_h$ and \eqref{reg2} by $\chi_h\in V_h$ and subtracting them from \eqref{eq:emacreg1}-\eqref{eq:emacreg3} to get
\begin{align}
\lip{u_t-(u_h)_t}{v_h}+c(w,w,v_h)-c(w_h,w_h,v_h)-\lip{P-q_h}{\divergence v_h}+\nu\lip{\nabla e^u}{\nabla v_h}=0&,\label{eq:diff1}\\
\alpha^2\lip{\nabla e^w}{\nabla \chi_h}+\lip{e^w}{\chi_h}=\lip{e^u}{\chi_h}&,\label{eq:diff2}
\end{align}
for any $v_h,\chi_h \in V_h$.  \textcolor{blue}{Note that  $\lip{-\lambda_h}{\divergence \chi_h}=0$ since $\chi_h \in V_h$.  We also obtain $\lip{P-q_h}{\divergence v_h}$ from
\begin{align*}
\lip{P-P_h}{\divergence v_h}=\lip{P-P_h}{\divergence v_h}+\lip{P_h-q_h}{\divergence v_h}=\lip{P-q_h}{\divergence v_h},
\end{align*}
which holds for any $q_h \in Q_h$, since $v_h \in V_h$.}

Now let $e^u=\phi_h^u+\eta^u$ and $e^w=\phi_h^w+\eta^w$ where $\eta^u=u-I_h^{St}(u)$, $\phi_h^u=I_h^{St}(u)-u_h$, $\eta^w=w-I_h^{St}(u)$, $\phi_h^w=I_h^{St}(w)-w_h$, and set $v_h=\phi_h^w$ and $\chi_h=\phi_h^w$.  Then after some manipulation \eqref{eq:diff1} becomes
\begin{align}
\begin{split}
\lip{(\phi_h^u)_t}{\phi_h^w}+\nu\lip{\nabla \phi_h^u}{\nabla \phi_h^w}=&-\lip{\eta^u_t}{\phi_h^w}-(c(w,w,\phi_h^w)\\
&-c(w_h,w_h,\phi_h^w))+\lip{P-q_h}{\divergence \phi_h^w}.\label{eq:diff1a}
\end{split}
\end{align}
\\
\textcolor{blue}{\textbf{Step 2:}  Finding suitable expressions for $\lip{\nabla \phi_h^u}{\nabla \phi_h^w}$ and  $\lip{(\phi_h^u)_t}{\phi_h^w}$ in equation \eqref{eq:diff1a}.}
\\
\\
\textcolor{blue}{First, we use the orthogonality of the Stokes operator \eqref{eq:StokesOp} to get}
\begin{align}
\begin{split}
\textcolor{blue}{\lip{\nabla \phi_h^w}{\nabla\eta^w}=\lip{\nabla \phi_h^w}{\nabla (w-I_h^{St}w)}=0,} \label{eq:grad_ortho}
\end{split}
\end{align}
\textcolor{blue}{since $\phi_h^w \in V_h$.  Next we test \eqref{eq:diff2} with $\chi_h=A_h\phi_h^w$ and use \eqref{eq:StokesProj} to get}
\begin{align*}
&\alpha^2\lip{\nabla e^w}{\nabla A_h\phi_h^w}+\lip{e^w}{A_h\phi_h^w}=\lip{e^u}{A_h\phi_h^w}.\notag
\end{align*}
We use $e^w=\phi_h^w+\eta^w$ to separate each term, \textcolor{blue}{then apply \eqref{eq:grad_ortho} to get}
\begin{align}
\lip{\nabla \phi_h^u}{\nabla \phi_h^w}=\alpha^2\norm{A_h\phi_h^w}^2+\norm{\nabla \phi_h^w}^2.\label{eq:new2}
\end{align}
\textcolor{blue}{Now we must handle $\lip{(\phi_h^u)_t}{\phi_h^w}$.}  We test \eqref{eq:diff2} with $\chi_h=\phi_h^w$ and apply \eqref{eq:StokesProj} to get
\begin{align}
\begin{split}
\lip{\phi_h^u}{\phi_h^w}+\lip{\eta^u}{\phi_h^w}&=\alpha^2\norm{\nabla\phi_h^w}^2+\norm{\phi_h^w}^2+\lip{\eta^w}{\phi_h^w}. \label{eq:poop}
\end{split}
\end{align}
Next, take the derivative with respect to $t$ of \eqref{eq:poop} and simplify,
\begin{align}
\lip{(\phi_h^u)_t}{\phi_h^w}&=\frac{d}{dt}[\alpha^2\norm{\nabla\phi_h^w}^2+\norm{\phi_h^w}^2+\lip{\eta^w}{\phi_h^w}-\lip{\eta^u}{\phi_h^w}]-\lip{\phi_h^u}{(\phi_h^w)_t}\notag\\
&=\frac{d}{dt}\brak{\alpha^2\norm{\nabla \phi_h^w}^2+\norm{\phi_h^w}^2}+\lip{\eta^w_t}{\phi_h^w}+\lip{\eta^w}{(\phi_h^w)_t}\notag\\
& \qquad -\lip{\eta^u_t}{\phi_h^w}-\lip{\eta^u}{(\phi_h^w)_t}-\lip{\phi_h^u}{(\phi_h^w)_t}. \label{eq:bigtimeterm1}
\end{align}
\\
\textcolor{blue}{\textbf{Step 3:} Finding a suitable expression for $\lip{\phi_h^u}{(\phi_h^w)_t}$ from \eqref{eq:bigtimeterm1}.}
\\
\\
\textcolor{blue}{Similar to what we did in \eqref{eq:grad_ortho}, we will do the same for $\lip{(\phi_h^w)_t}{\eta^w}$ and use \eqref{eq:StokesOp} to get}
\begin{align}
\begin{split}
\textcolor{blue}{\lip{\nabla(\phi_h^w)_t}{\nabla \eta^w}=\lip{\nabla(\phi_h^w)_t}{\nabla(w_t-I_h^{St}w)}=0}.\label{eq:grad_ortho_time}
\end{split}
\end{align}
\textcolor{blue}{Now we test \eqref{eq:diff2} with $\chi_h=(\phi_h^w)_t$ and use \eqref{eq:grad_ortho_time} to get}
\begin{align}
\begin{split}
\lip{\phi_h^u}{(\phi_h^w)_t}&=\frac{\alpha^2}{2}\frac{d}{dt}\norm{\nabla \phi_h^w}^2+\frac{1}{2}\frac{d}{dt}\norm{\phi_h^w}^2+\lip{\eta^w}{(\phi_h^w)_t}-\lip{\eta^u}{(\phi_h^w)_t}. \label{eq:bigtimeterm2}
\end{split}
\end{align}
Substituting \eqref{eq:bigtimeterm2} into \eqref{eq:bigtimeterm1} and rearranging yields
\begin{align}
\begin{split}
\lip{(\phi_h^u)_t}{\phi_h^w}&=\frac{1}{2}\frac{d}{dt}\brak{\alpha^2\norm{\nabla \phi_h^w}^2+\norm{\phi_h^w}^2}+\lip{\eta^w_t}{\phi_h^w}-\lip{\eta^u_t}{\phi_h^w}. \label{eq:new1}
\end{split}
\end{align}
Now combine \eqref{eq:new1} and \eqref{eq:new2} to get the identity
\begin{align}
\begin{split}
&\lip{(\phi_h^u)_t}{\phi_h^w}+\nu\lip{\nabla \phi_h^u}{\nabla \phi_h^w}\\
=&\frac{1}{2}\frac{d}{dt}\brak{\alpha^2\norm{\nabla \phi_h^w}^2+\norm{\phi_h^w}^2}+\lip{\eta^w_t}{\phi_h^w}-\lip{\eta^u_t}{\phi_h^w}+\nu\brak{\alpha^2\norm{A_h\phi_h^w}^2+\norm{\nabla \phi_h^w}^2}. \label{eq:bigassidentity}
\end{split}
\end{align}
Rearranging \eqref{eq:bigassidentity} now provides
\begin{align}
\begin{split}
&\frac{1}{2}\frac{d}{dt}\brak{\alpha^2\norm{\nabla \phi_h^w}^2+\norm{\phi_h^w}^2}+\nu\brak{\alpha^2\norm{A_h\phi_h^w}^2+\norm{\nabla \phi_h^w}^2}\\
=&\lip{(\phi_h^u)_t}{\phi_h^w}+\nu\lip{\nabla \phi_h^u}{\nabla \phi_h^w}-\lip{\eta^w_t}{\phi_h^w}+\lip{\eta^u_t}{\phi_h^w}. \label{eq:bigguy}
\end{split}
\end{align}
Putting together \eqref{eq:diff1a} and \eqref{eq:bigguy} and canceling the $\lip{\eta^u_t}{\phi_h^w}$ term yields
\begin{align}
\begin{split}
&\frac{1}{2}\frac{d}{dt}\brak{\alpha^2\norm{\nabla \phi_h^w}^2+\norm{\phi_h^w}^2}+\nu\brak{\alpha^2\norm{A_h\phi_h^w}^2+\norm{\nabla \phi_h^w}^2}\\
=&-(c(w,w,\phi_h^w)-c(w_h,w_h,\phi_h^w))+\lip{P-q_h}{\divergence \phi_h^w}-\lip{\eta^w_t}{\phi_h^w}. \label{eq:lastequality}
\end{split}
\end{align}
\\
\textcolor{blue}{\textbf{Step 4:} Estimates for each term in equation \eqref{eq:lastequality}.}
\\
\\
Estimates for the right hand side of \eqref{eq:lastequality} where we employ the Cauchy-Schwarz inequality and Young's inequality are as follows:
\begin{align}
&|\lip{\eta^w_t}{\phi_h^w}|\leq 2\nu\inv\norm{\eta^w_t}^2_{X'}+\frac{\nu}{2}\norm{\phi_h^w}^2,\label{eq:finalbounds1}\\
&|\lip{P-q_h}{\divergence \phi_h^w}|\leq 2\nu\inv \norm{P-q_h}^2+\frac{\nu}{2}\norm{\nabla \phi_h^w}^2.\label{eq:finalbounds2}
\end{align}
From Theorem 3.2 in \cite{OR20}, we have the following bound for the nonlinear terms,
\begin{align*}
\begin{split}
&|c(w,w,\phi_h^w)-c(w_h,w_h,\phi_h^w)|\\
& \hspace{20mm} \leq C\pare{\norm{\eta^w}^2\linfnorm{w}+\linfnorm{\nabla w}\norm{\eta^w}^2+\brak{\linfnorm{w}+\linfnorm{\nabla w}}\norm{\phi_h^w}^2}. 
\end{split}
\end{align*}
Applying \eqref{eq:finalbounds1}-\eqref{eq:finalbounds2} to \eqref{eq:lastequality} gives us
\begin{align*}
\begin{split}
&\frac{1}{2}\frac{d}{dt}\brak{\alpha^2\norm{\nabla \phi_h^w}^2+\norm{\phi_h^w}^2}+\nu\brak{\alpha^2\norm{A_h\phi_h^w}^2+\norm{\nabla \phi_h^w}^2}\\
\leq &C\pare{\norm{\eta^w}^2\linfnorm{w}+\linfnorm{\nabla w}\norm{\eta^w}^2+\brak{\linfnorm{w}+\linfnorm{\nabla w}}\norm{\phi_h^w}^2}\\
&+2\nu\inv\pare{\norm{\eta^w}^2_{X'}+\norm{P-q_n}^2}+\frac{\nu}{2}\norm{\phi_h^w}^2+\frac{\nu}{2}\norm{\nabla \phi_h^w}^2,
\end{split}
\end{align*}
and then by rearranging we get
\begin{align*}
\begin{split}
&\frac{d}{dt}\brak{\alpha^2\norm{\nabla \phi_h^w}^2+\norm{\phi_h^w}^2}+\nu\brak{2\alpha^2\norm{A_h\phi_h^w}^2+\norm{\nabla \phi_h^w}^2}\\
\leq & 2C\pare{\norm{\eta^w}^2\linfnorm{w}+\linfnorm{\nabla w}\norm{\eta^w}^2}+2\pare{C\brak{\linfnorm{w}+\linfnorm{\nabla w}}+\nu}\norm{\phi_h^w}^2\\
&+4\nu\inv\pare{\norm{\eta_t^w}^2_{X'}+\norm{P-q_n}^2}.
\end{split}
\end{align*}
Finally apply Gronwall's inequality and get
\begin{align*}
\begin{split}
&\alpha^2\norm{\nabla \phi_h^w}^2+\norm{\phi_h^w}^2+\nu\int_0^T\brak{2\alpha^2\norm{A_h\phi_h^w}^2+\norm{\nabla \phi_h^w}^2}dt\\
\leq& K\left[ 4\nu\inv\left(\norm{\eta^w_t}^2_{L^2(0,T;X')} +\inf_{q_h \in L^2(0,T;Q_h)}\norm{P-q_h}^2_{L^2(0,T;L^2)}\right) \right],
\end{split}
\end{align*}
where
\begin{align*}
K=\exp\brak{2C\pare{\norm{w}_{L^1(0,T;L^{\infty})}+\norm{\nabla w}_{L^1(0,T;L^{\infty})}}+T\nu}.
\end{align*}
\textcolor{blue}{By the triangle inequality, we get the desired error term}
\begin{align}
\begin{split}
&\alpha^2\norm{\nabla e^w}^2+\norm{e^w}^2+\nu\int_0^T\brak{2\alpha^2\norm{A_he^w}^2+\norm{\nabla e^w}^2}dt\\
\leq & \alpha^2\norm{\nabla \eta^w}^2+\norm{\eta^w}^2+\nu\brak{2\alpha^2\norm{A_h\eta^w}^2_{L^2(0,T;L^2)}+\norm{\nabla \eta^w}^2_{L^2(0,T;L^2)}}\\
&+K\left[ 4\nu\inv\left(\norm{\eta^w_t}^2_{L^2(0,T;X')}+\inf_{q_h \in L^2(0,T;Q_h)}\norm{P-q_h}^2_{L^2(0,T;L^2)}\right) \right]. \label{eq:3.4bound}
\end{split}
\end{align}
\\
\textcolor{blue}{\textbf{Step 5:} Use the bound on $e^w$ to find a bound for $e^u$.}
\\
\\
\textcolor{blue}{
Set $\chi_h=\phi_h^u$ in \eqref{eq:emacreg3} and expand $e^u$ to get
\begin{align*}
\norm{\phi_h^u}^2=\alpha^2\lip{\nabla e^w}{\nabla \phi_h^u}+\lip{e^w}{\phi_h^u}-\lip{\eta^u}{\phi_h^u}.
\end{align*}
Applying the Cauchy-Schwarz inequality gives
\begin{align*}
\norm{\phi_h^u}^2\leq \alpha^2\norm{\nabla e^w}\norm{\nabla \phi_h^u}+\norm{e^w}\norm{\phi_h^u}+\norm{\eta^u}\norm{\phi_h^u}.
\end{align*}
Now use the inverse inequality on the first term of the left hand side,
\begin{align*}
\norm{\phi_h^u}^2\leq C\alpha^2\frac{1}{h}\norm{\nabla e^w}\norm{\phi_h^u}+\norm{e^w}\norm{\phi_h^u}+\norm{\eta^u}\norm{\phi_h^u}.
\end{align*}
Recall from Remark 3.1 that $\alpha=O(h)$, and then we reduce to get
\begin{align*}
\norm{\phi_h^u}\leq C\pare{\alpha \norm{\nabla e^w}+\norm{e^w}}+\norm{\eta^w}.
\end{align*}
Squaring both sides and using Young's inequality on the right hand side gives us
\begin{align*}
\norm{\phi_h^u}^2\leq 2C\pare{\alpha^2\norm{\nabla e^w}^2+\norm{e^w}^2}+2\norm{\eta^u}^2.
\end{align*}
For our error bound, we get
\begin{align*}
\norm{e^u}\leq 2C\pare{\alpha^2\norm{\nabla e^w}^2+\norm{e^w}^2}+2\norm{\eta^u}^2.
\end{align*}
Now applying \eqref{eq:3.4bound} and \eqref{eq:etabound} gives us the stated result.
}
\end{proof}

\section{Numerical experiments}

In this section, we provide results for several numerical experiments that test EMAC-Reg and compare it to NSE schemes and other related models.  Define $\alpha$ to be the filtering radius of the Helmholtz filter, typically chosen as $\alpha=O(h)$.   The schemes that we utilize in addition to EMAC-Reg are as follows. Find $(u_h,p_h) \in (X_h,Q_h)$ such that for every $(v_h,q_h)\in (X_h,Q_h)$:

\vspace{3mm}

NSE Skew-symmetric (SKEW)\\ 
Find $(u_h,p_h) \in (X_h,Q_h)$ such that for every $(v_h,q_h)\in (X_h,Q_h)$:
\begin{align*}
\lip{(u_h)_t}{v_h}+\lip{u_h \cdot \nabla u_h}{v_h}+\frac{1}{2}\lip{(\divergence u_h)u_h}{v_h}-\lip{p_h}{\divergence v_h}+\nu\lip{\nabla u_h}{\nabla v_h}&=\lip{f}{v_h},\\
\lip{\divergence u_h}{q_h}&=0.
\end{align*}

NSE EMAC (EMAC)\\ 
Find $(u_h,p_h) \in (X_h,Q_h)$ such that for every $(v_h,q_h)\in (X_h,Q_h)$:
\begin{align*}
\lip{(u_h)_t}{v_h}+2\lip{D(u_h)u_h}{v_h}+\lip{(\divergence u_h)u_h}{v_h}-\lip{p_h}{\divergence v_h}+\nu\lip{\nabla u_h}{\nabla v_h}&=\lip{f}{v_h},\\
\lip{\divergence u_h}{q_h}&=0.
\end{align*}

NS-$\alpha$\\
Find $(u_h,p_h,w_h,\lambda_h) \in (X_h,Q_h,X_h,Q_h)$ such that for every\\ $(v_h,q_h,\chi_h,r_h)\in (X_h,Q_h,X_h,Q_h)$:
\begin{align*}
\lip{(u_h)_t}{v_h}+\lip{u_h \times w_h}{v_h}-\lip{p_h}{\nabla \cdot v_h}+\nu\lip{\nabla u_h}{\nabla v_h}&=\lip{f}{v_h},\\
\lip{\nabla \cdot u_h}{q_h}&=0,\\
\lip{\lambda_h}{\nabla \cdot \chi_h}+\alpha^2\lip{\nabla w_h}{\nabla \chi_h}+\lip{w_h}{\chi_h}&=\lip{u_h}{\chi_h},\\
\lip{\nabla \cdot w_h}{r_h}&=0.
\end{align*}

\subsection{Convergence rate test for a problem with analytical solution}

We now test the convergence results described by Theorem \ref{error_theorem} on a problem that is analogous to the Chorin problem in \cite{Cho68}, however it is slightly adjusted for the EMAC-Reg model.  We deduce that for $\alpha \geq 0$, the filtered velocity, velocity and pressure defined by
\begin{align*}
w &= \bmat{-\cos(\pi x)\sin(\pi y)\\ \sin(\pi x)\cos(\pi y)}e^{-2\pi^2\nu t},\\
u &= (1+2\pi^2\alpha^2)w,\\
p &= -w \cdot \nabla w,
\end{align*}
is a solution for \eqref{reg1}-\eqref{reg2} with $f=0$ and initial conditions $w_0=w(0), \: u_0=u(0)$.  We consider the domain $\Omega=(0,1)^2$ and enforce the appropriate Dirichlet boundary conditions for $u$ and $w$.

\textcolor{blue}{We will be conducting three experiments: a spatial convergence test with $\Delta t$ fixed, a temporal convergence test with $h$ fixed, and then a hybrid of the two.  For each test, we take $\alpha=\frac{h}{2}$, $\nu=.2$ and use $(P_2,P_1)$ elements on a uniform mesh.  For the spatial convergence test, we use $\Delta t=.005$ with Crank-Nicolson time stepping and fix our end time at $T=1$.  We calculate the error with $h=\pare{\frac{1}{2}}^i$ where $i=1,\dots,7$.}

\textcolor{blue}{For the temporal convergence test, we have nearly identical parameters as the spatial convergence test.  We set $h=\frac{1}{128}$ and $\Delta t=2^{-i}$ for $i=0,\dots,6$.}

\textcolor{blue}{For the hybrid experiment, we will use incrementing values for both $h$ and $\Delta t$.  We will start with $h=\frac{1}{2}$ and $\Delta t=1$, then $h=\frac{1}{4}$ and $\Delta t=\frac{1}{2}$, and so on.  For each of these experiments, we will be calculating $L^{\infty}(0,T;L^2)$ and $L^{\infty}(0,T;H^1)$}

\textcolor{blue}{We expect second order convergence for the $H^1$ norm, and an $L^2$ lift is expected to make the $L^2$ convergence third order for both tests.  The results are displayed in Tables \ref{fig:convergence}, \ref{fig:temporal_convergence}, and \ref{fig:temporal_convergence} for spatial, temporal, and the combined convergence respectively.}  We observe third order accuracy in the $L^2$ norm and second order accuracy in the $H^1$ norm \textcolor{blue}{for spatial convergence, as expected.  For temporal convergence, we observe second order convergence at the least for $L^2$ and mixed convergence rates for $H^1$.  We get high convergence rates starting around $\Delta t=\frac{1}{4}$ through $\Delta t=\frac{1}{32}$, but both maintain at least second order convergence.  We get similar results for the hybrid experiment.}

\begin{table}[H]
\begin{tabular}{|c||c|c||c|c||c|c||}
\hline
\textcolor{blue}{h} & 
\textcolor{blue}{$\norm{w-w_h}_{\infty,0}$} & 
\textcolor{blue}{Rate} & 
\textcolor{blue}{$\norm{w-w_h}_{\infty,1}$} & 
\textcolor{blue}{Rate} & 
\textcolor{blue}{$\norm{u-u_h}_{\infty,0}$} & 
\textcolor{blue}{Rate} \\
\hline
\textcolor{blue}{1/2} & 
\textcolor{blue}{8.98240e-04} & 
\textcolor{blue}{-} & 
\textcolor{blue}{1.33377e-02} & 
\textcolor{blue}{-} & 
\textcolor{blue}{8.98240e-04} & 
\textcolor{blue}{-} \\
\hline
\textcolor{blue}{1/4} & 
\textcolor{blue}{1.07331e-04} & 
\textcolor{blue}{3.06503} & 
\textcolor{blue}{3.55447e-03} & 
\textcolor{blue}{1.90780} & 
\textcolor{blue}{1.07331e-03} & 
\textcolor{blue}{3.06503}\\
\hline
\textcolor{blue}{1/8} & 
\textcolor{blue}{1.30963e-05} & 
\textcolor{blue}{3.03484} & 
\textcolor{blue}{9.11945e-04} & 
\textcolor{blue}{1.96262} & 
\textcolor{blue}{1.30963e-05} & 
\textcolor{blue}{3.03484}\\
\hline
\textcolor{blue}{1/16} & 
\textcolor{blue}{1.62923e-06} & 
\textcolor{blue}{3.00689} & 
\textcolor{blue}{2.29787e-04} & 
\textcolor{blue}{1.98865} & 
\textcolor{blue}{1.62923e-06} & 
\textcolor{blue}{3.00689}\\
\hline
\textcolor{blue}{1/32} & 
\textcolor{blue}{2.04701e-07} & 
\textcolor{blue}{2.99260} & 
\textcolor{blue}{5.75694e-05} & 
\textcolor{blue}{1.99692} & 
\textcolor{blue}{2.04701e-07} & 
\textcolor{blue}{2.99260}\\
\hline
\textcolor{blue}{1/64} & 
\textcolor{blue}{3.32277e-08} & 
\textcolor{blue}{2.62306} & 
\textcolor{blue}{1.47046e-05} & 
\textcolor{blue}{1.96904} & 
\textcolor{blue}{3.32277e-08} & 
\textcolor{blue}{2.62306}\\
\hline
\textcolor{blue}{1/128} & 
\textcolor{blue}{2.10578e-08} & 
\textcolor{blue}{0.65803} & 
\textcolor{blue}{4.66998e-06} & 
\textcolor{blue}{1.65478} & 
\textcolor{blue}{2.10578e-08} & 
\textcolor{blue}{0.65803}\\
\hline
\end{tabular}
\caption{\textcolor{blue}{Spatial convergence results for both $u$ and $w$ for EMAC-Reg}}
\label{fig:convergence}
\end{table}

\begin{table}[H]
\begin{tabular}{|c||c|c||c|c||c|c||}
\hline
\textcolor{blue}{$\Delta t$} 
& \textcolor{blue}{$\norm{w-w_h}_{\infty,0}$} 
& \textcolor{blue}{Rate} 
& \textcolor{blue}{$\norm{w-w_h}_{\infty,1}$} 
& \textcolor{blue}{Rate} 
& \textcolor{blue}{$\norm{u-u_h}_{\infty,0}$} 
& \textcolor{blue}{Rate} \\
\hline
\textcolor{blue}{1} 
& \textcolor{blue}{1.36516e-02}
& \textcolor{blue}{-} 
& \textcolor{blue}{1.70370e-01} 
& \textcolor{blue}{-} 
& \textcolor{blue}{1.36516e-02} 
& \textcolor{blue}{-} \\
\hline
\textcolor{blue}{1/2} 
& \textcolor{blue}{3.99355e-03} 
& \textcolor{blue}{1.77333} 
& \textcolor{blue}{5.28610e-01} 
& \textcolor{blue}{1.68840} 
& \textcolor{blue}{3.99355e-03} 
& \textcolor{blue}{1.77333}\\
\hline
\textcolor{blue}{1/4} 
& \textcolor{blue}{2.94349e-04} 
& \textcolor{blue}{3.76207} 
& \textcolor{blue}{6.19844e-03} 
& \textcolor{blue}{3.09223}
& \textcolor{blue}{2.94349e-04} 
& \textcolor{blue}{3.76207}\\
\hline
\textcolor{blue}{1/8} 
& \textcolor{blue}{1.78485e-05} 
& \textcolor{blue}{4.04366} 
& \textcolor{blue}{5.79085e-04} 
& \textcolor{blue}{3.42006} 
& \textcolor{blue}{1.78485e-05} 
& \textcolor{blue}{4.04366}\\
\hline
\textcolor{blue}{1/16} 
& \textcolor{blue}{3.24554e-06} 
& \textcolor{blue}{2.45927} 
& \textcolor{blue}{6.32790e-05} 
& \textcolor{blue}{3.19398} 
& \textcolor{blue}{3.24554e-06} 
& \textcolor{blue}{2.45927}\\
\hline
\textcolor{blue}{1/32} 
& \textcolor{blue}{8.03788e-07} 
& \textcolor{blue}{2.01357} 
& \textcolor{blue}{1.29501e-05} 
& \textcolor{blue}{2.28876} 
& \textcolor{blue}{8.03788e-07} 
& \textcolor{blue}{2.01357}\\
\hline
\textcolor{blue}{1/64} 
& \textcolor{blue}{2.01230e-07} 
& \textcolor{blue}{1.99797} 
& \textcolor{blue}{7.41362e-06} 
& \textcolor{blue}{0.80472} 
& \textcolor{blue}{2.01230e-07} 
& \textcolor{blue}{1.99797}\\
\hline
\end{tabular}
\caption{\textcolor{blue}{Temporal convergence results for both $u$ and $w$ for EMAC-Reg}}
\label{fig:temporal_convergence}
\end{table}

\begin{table}[H]
\begin{tabular}{|c|c||c|c||c|c||c|c||}
\hline
\textcolor{blue}{h} 
& \textcolor{blue}{$\Delta t$} 
& \textcolor{blue}{$\norm{w-w_h}_{\infty,0}$} 
& \textcolor{blue}{Rate} 
& \textcolor{blue}{$\norm{w-w_h}_{\infty,1}$} 
& \textcolor{blue}{Rate} 
& \textcolor{blue}{$\norm{u-u_h}_{\infty,0}$} 
& \textcolor{blue}{Rate} \\
\hline
\textcolor{blue}{1/2} 
& \textcolor{blue}{1} 
& \textcolor{blue}{9.08829e-03} 
& \textcolor{blue}{-} 
& \textcolor{blue}{1.27869e-01} 
& \textcolor{blue}{-} 
& \textcolor{blue}{9.08829e-03} 
& \textcolor{blue}{-} \\
\hline
\textcolor{blue}{1/4} 
& \textcolor{blue}{1/2} 
& \textcolor{blue}{6.37691e-03} 
& \textcolor{blue}{0.51115} 
& \textcolor{blue}{0.11360e-00} 
& \textcolor{blue}{0.17070} 
& \textcolor{blue}{6.37691e-03} 
& \textcolor{blue}{0.51115} \\
\hline
\textcolor{blue}{1/8} 
& \textcolor{blue}{1/4} 
& \textcolor{blue}{4.71124e-04} 
& \textcolor{blue}{3.75868} 
& \textcolor{blue}{2.00927e-02} 
& \textcolor{blue}{2.49922} 
& \textcolor{blue}{4.71124e-04} 
& \textcolor{blue}{3.75868}  \\
\hline
\textcolor{blue}{1/16} 
& \textcolor{blue}{1/8} 
& \textcolor{blue}{2.99146e-05} 
& \textcolor{blue}{3.97719} 
& \textcolor{blue}{2.89569e-03} 
& \textcolor{blue}{2.79469} 
& \textcolor{blue}{2.99146e-05} 
& \textcolor{blue}{3.97719}  \\
\hline
\textcolor{blue}{1/32} 
& \textcolor{blue}{1/16} 
& \textcolor{blue}{3.59222e-06} 
& \textcolor{blue}{3.05790} 
& \textcolor{blue}{3.83564e-04} 
& \textcolor{blue}{2.91637} 
& \textcolor{blue}{3.59222e-06} 
& \textcolor{blue}{3.05790}  \\
\hline
\textcolor{blue}{1/64} 
& \textcolor{blue}{1/32} 
& \textcolor{blue}{8.09962e-07} 
& \textcolor{blue}{2.14895} 
& \textcolor{blue}{5.08696e-05} 
& \textcolor{blue}{2.91459} 
& \textcolor{blue}{8.09962e-07} 
& \textcolor{blue}{2.14895}  \\
\hline
\textcolor{blue}{1/128} 
& \textcolor{blue}{1/64} 
& \textcolor{blue}{2.01230e-07} 
& \textcolor{blue}{2.00901} 
& \textcolor{blue}{7.41362e-06} 
& \textcolor{blue}{2.77855} 
& \textcolor{blue}{2.01230e-07} 
& \textcolor{blue}{2.00901}  \\
\hline
\end{tabular}
\caption{\textcolor{blue}{Convergence results for decreasing values of $h$ and $\Delta t$ for $u$ and $w$ for EMAC-Reg}}
\label{fig:hybrid_convergence}
\end{table}

\subsection{Gresho problem}

Our next experiment is the Gresho Problem, which is also referred to as the standing vortex problem.  We start with an initial condition $u_0$ which is a solution of the steady Euler equations.  Define $r=\sqrt{x^2+y^2}$, and on $\Omega=(-0.5,0.5)^2$, the velocity and pressure solutions are defined by
\begin{align*}
&r\leq .2:\begin{cases}
&u=\bmat{-5y\\5x}\\
&p=12.5r^2+C_1
\end{cases}\\
&.2\leq r\leq .4: \begin{cases}
&u=\bmat{\frac{2y}{r}+5y\\\frac{2x}{r}-5x}\\
&p=12.5r^2-20r+4\log(r)+C_2
\end{cases}\\
&r>.4:\begin{cases}
&u=\bmat{0\\0}\\
&p=0
\end{cases}
\end{align*}
where
\begin{align*}
&C_2=(-12.5)(.4)^2+20(.4)^2-4\log(.4),\\
&C_1=C_2-20(.2)+4\log(.2).
\end{align*}
We compute solutions using EMAC-Reg, EMAC, NS-$\alpha$, and SKEW with Crank-Nicolson time stepping and Newton iterations to solve the nonlinear problem, with $f=0$ and no slip boundary conditions up to $T=4.0$, \textcolor{blue}{and we take $\nu=0$ (we also ran using $\nu=10^{-7}$ and obtained the same results)}.  We computed using $(P_2,P_1)$ Taylor-Hood elements on a $48 \times 48$ uniform mesh, with a time step of $\Delta t=0.01$.  A spatial radius of $\alpha=\frac{1}{50}$ was used for EMAC-Reg and NS-$\alpha$.  Other values of $\alpha$ that were $O(h)$ were tested and gave similar results.

\begin{figure}[H]
\centering
\includegraphics[scale=.2]{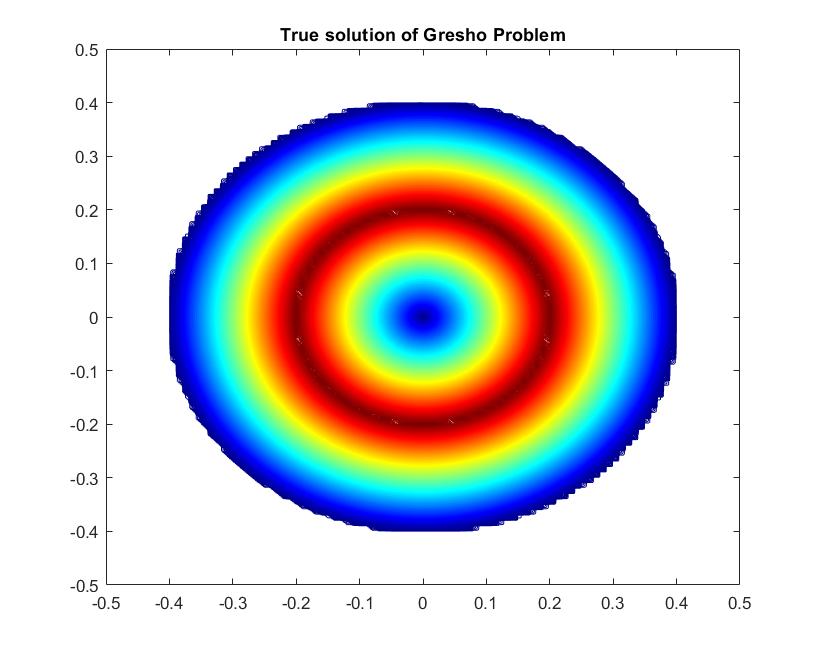}
\caption{Speed contours of the true solution of the Gresho problem at all times.}
\label{fig:TrueGresho}
\end{figure}

\begin{figure}[H]
\centering
\includegraphics[scale=.3]{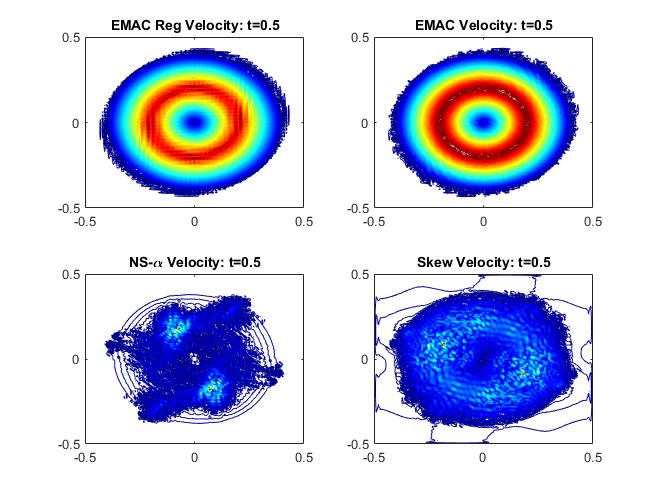}
\includegraphics[scale=.3]{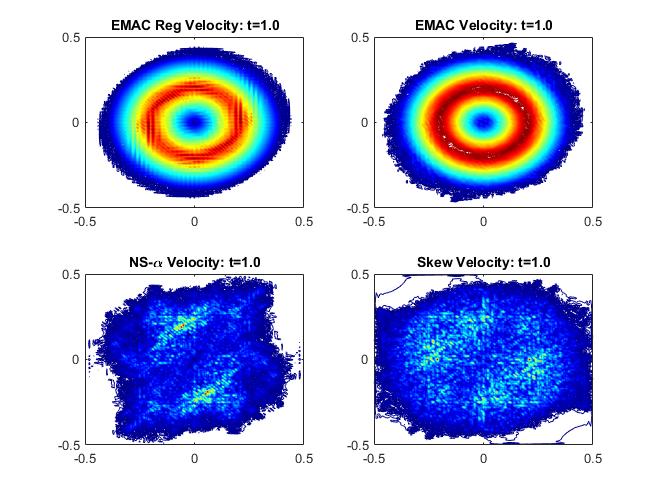}
\includegraphics[scale=.3]{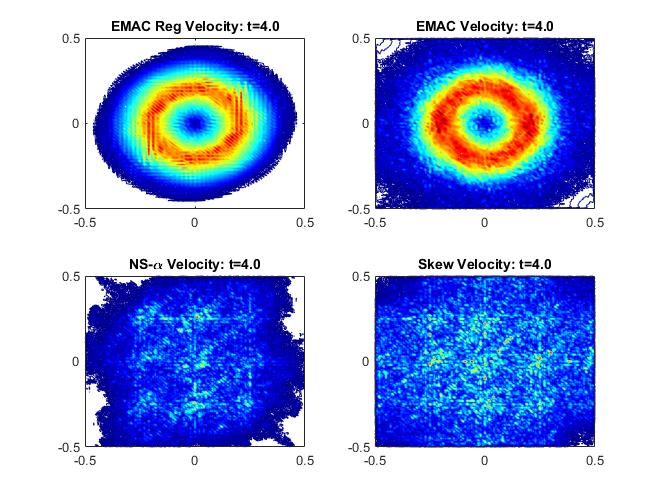}
\caption{Speed contours of our schemes at times $t=.5$, $1.0$, and $4.0$.}
\label{fig:Gresho}
\end{figure}

The Gresho problem solution is constant in time and so as time goes forward a good numerical solution will look similar to figure \ref{fig:TrueGresho}, which is the initial condition plot of speed.

At time $t=.5$, we observe from figure \ref{fig:Gresho} that EMAC-Reg and EMAC perform about equally well, whereas NS-$\alpha$ and SKEW are nearly unrecognizable compared to the true solution in figure \ref{fig:TrueGresho}.  At time $t=1$, we start to see EMAC-Reg outperform EMAC: EMAC-Reg approximately maintains the vortex shape, whereas the vortex for EMAC is spreading out with rough edges.  At time $t=4$, we see both vortices are spreading out, but it is clear that EMAC-Reg resembles figure \ref{fig:TrueGresho} much better.

We also compare $L^2$ error and predictions of important physical quantities that we hope to be preserved with each of these formulations: energy, momentum, and angular momentum versus time.  We have shown in Theorem \ref{Conservation_Thm} that EMAC-Reg preserves energy, momentum, and angular momentum and it can be shown that SKEW and NS-$\alpha$ preserve energy \cite{Layton08, LR12}. 

\begin{figure}[H]
\center
\includegraphics[scale=.25]{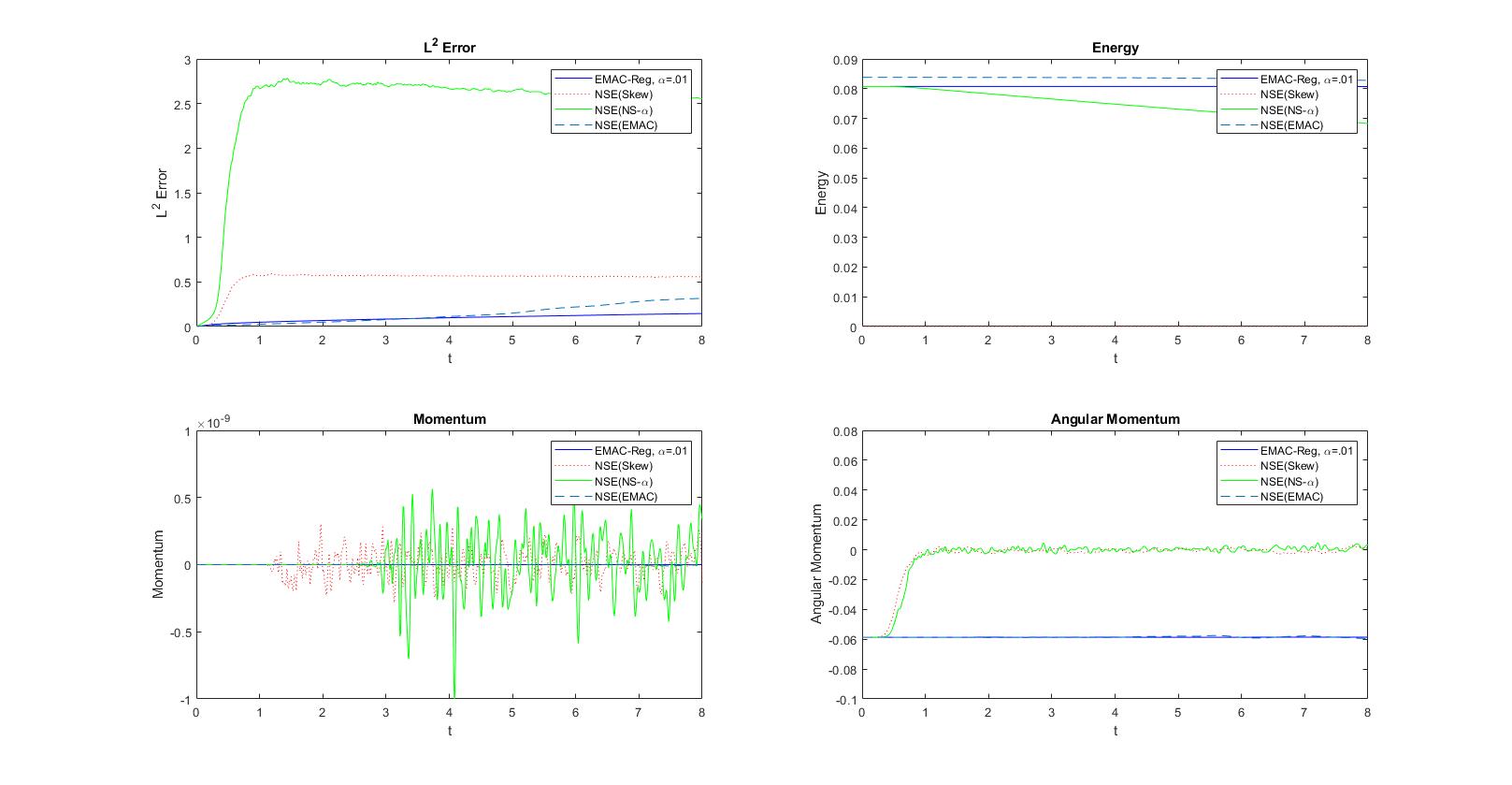}
\caption{Plots of time versus $L^2$ error, energy, momentum and angular momentum}
\label{quant_compare}
\end{figure}

From figure \ref{quant_compare} we observe that EMAC-Reg not only conserves energy, momentum, and angular momentum, but has significantly less $L^2$ error than the other formulations.  We note also that EMAC was shown in \cite{CHOR17} to have smaller $L^2$ error compared to several other formulations.  We also observe that each of these schemes conserves energy, as expected.  Momentum also appears to be conserved, noting the scale of the plot is $O(10^{-8})$.  SKEW and NS-$\alpha$ appear to have large variation on a smaller scale, whereas EMAC and EMAC-Reg stay constant over time.  

For angular momentum, as expected, SKEW and NS-$\alpha$ do not have constant angular momentum over time.  EMAC approximately conserves angular momentum, and EMAC-Reg does the best job of conserving it.

\subsection{2D Channel flow over a step}

The next experiment is flow past a forward-backward facing step.  We consider a 40 by 10 channel with a 1 by 1 step placed 5 units into the channel.  We enforce no-slip boundary conditions on the walls and the step, and use a constant (with respect to time) parabolic inflow and outflow with peak velocity 1, viscosity $\nu=\frac{1}{600}$, and no external force, i.e., $f=0$.

We ran this experiment using SKEW, EMAC, EMAC-Reg, and NS-$\alpha$.  We also ran SKEW on a fine mesh for reference, and note this solution matches the reference solution in \cite{LMNR08}.  For each model/scheme we used \textcolor{blue}{$(P_2,P_1)$ elements, Crank-Nicolson time stepping}, and Newton iterations to solve for the nonlinear term.  We used a timestep of $\Delta t=.025$ and ran until the end time $T=40$.  A spatial filter of $\alpha=\frac{1}{10}$ was used for EMAC-Reg and NS-$\alpha$.  There are 3810 degrees of freedom on the coarse mesh and 145K degrees of freedom on the reference mesh, and hence the coarse mesh simulations represent significantly underresolved simulations.

\begin{figure}[H]
\centering
\includegraphics[scale=.22]{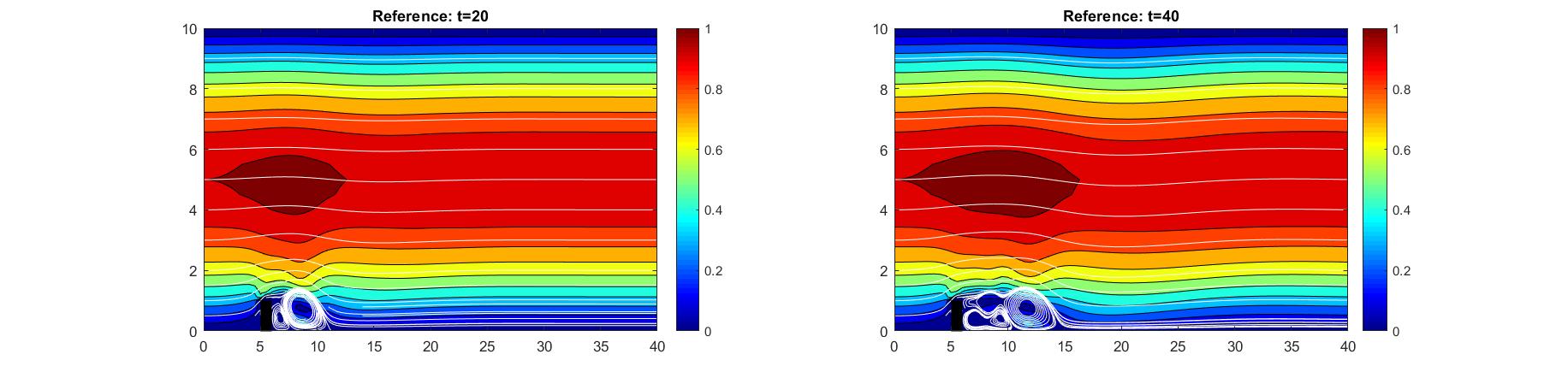}
\includegraphics[scale=.22]{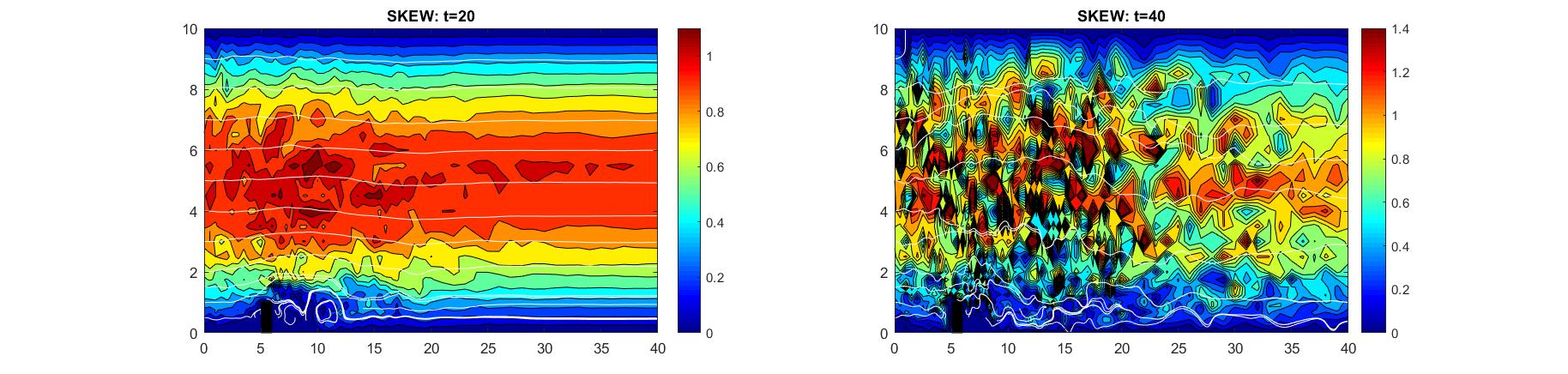}
\includegraphics[scale=.22]{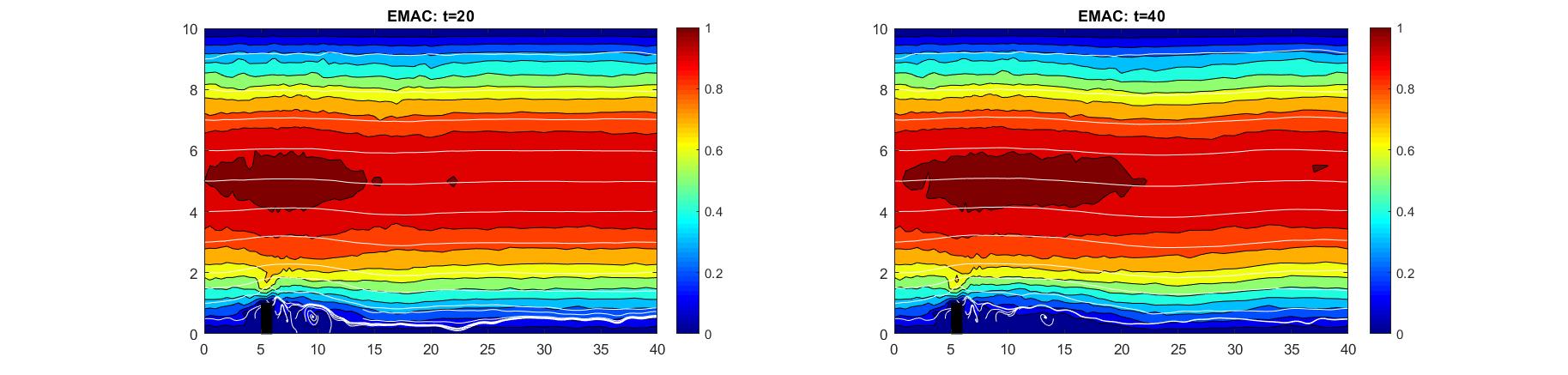}
\includegraphics[scale=.22]{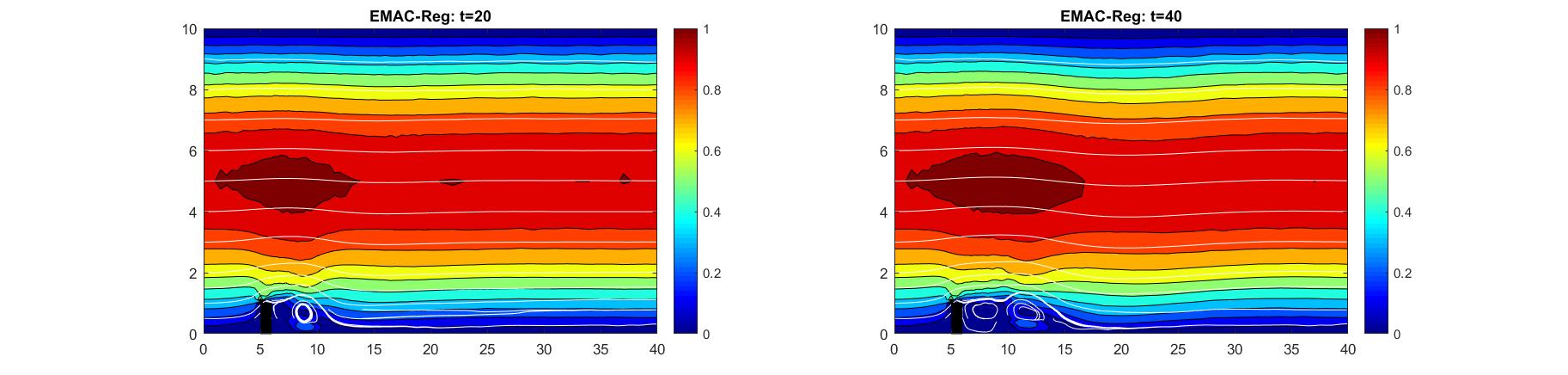}
\includegraphics[scale=.22]{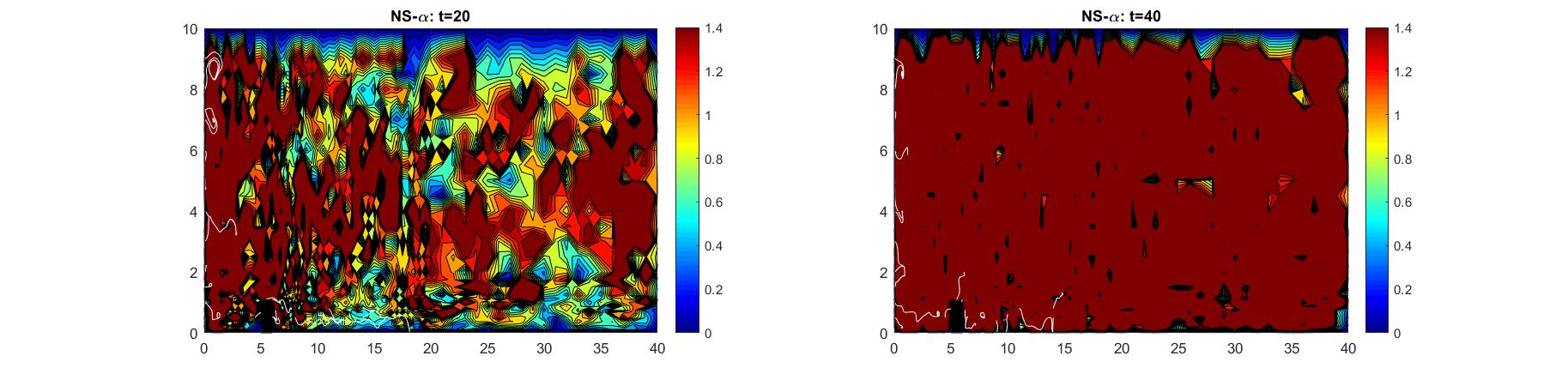}
\caption{Shown above are speed contours of our 4 different schemes' solution plotted at times t=20 and t=40 with SKEW on a fine mesh as reference}
\end{figure}
SKEW provides an accurate and smooth solution on the fine mesh that matches results from \cite{LR12} and we use as a reference solution.  However it does not perform well on the coarse mesh.

EMAC provides a better solution than SKEW on the coarse mesh, but still exhibits oscillations and does not predict eddy detachment.  EMAC-Reg provides the best coarse mesh solution, and qualitatively matches the reference solution quite well.  The NS-$\alpha$ coarse mesh solution is very poor.

\subsection{2D Kelvin-Helmholtz}

Our final numerical test is a benchmark problem from \cite{SJLLLS18} for 2D Kelvin-Helmholtz instability.  Our domain is the unit square with periodic boundary conditions at $x=0, \: 1$.  The no penetration boundary condition $u \cdot n = 0$ is strongly enforced at $y=0, \: 1$.  Weak enforcement of the free-slip condition is included as well.  The initial condition is defined as
\begin{align*}
u_0(x,y) = \bmat{u_{\infty}\tanh\pare{2y-1	\delta_0}\\ 0}+c_n\bmat{\partial_y \phi(x,y)\\ -\partial_x \phi(x,y)},
\end{align*}
where $\delta_0=\frac{1}{28}$ is the initial vorticity thickness, the reference velocity $u_{\infty}=1$, $c_n=10^{-3}$ the noise/scaling factor, and
\begin{align*}
\phi(x,y)=u_{\infty}\exp\pare{-\frac{(y-0.5)^2}{\delta_0^2}}\pare{\cos(8\pi x)  + \cos(20 \pi x)}.
\end{align*}
The Reynolds number in this case is defined by $Re=\frac{\delta_o u_{\infty}}{\nu}=\frac{1}{28\nu}$.

\textcolor{blue}{We compute solutions for $Re=1000$ using EMAC and EMAC-Reg; NS-$\alpha$ and SKEW} simulations using similar parameter choices as EMAC-Reg became unstable, and we were unable to compute to the endtime, even trying several choices of $\alpha$.  It is possible that the better results would be found with stabilizations \cite{C10} or divergence free elements \cite{BNP18}.  Taylor-Hood $(P_2,P_1)$ elements are used for spatial discretization and BDF2 for temporal discretization with a uniform triangulation with $h=\frac{1}{48}$ for \textcolor{blue}{EMAC and EMAC-Reg, which gave us 19K velocity degrees of freedom.  Note that we do lose some degree of property conservation by using BDF2 instead of Crank-Nicolson.  We also computed a reference solution, which was computed using EMAC with $h=\frac{1}{128}$.}  \textcolor{blue}{For EMAC-Reg we computed using $\alpha=\frac{h}{3}$.  We calculated up to $T=10$} using step size $\Delta t=.001$.  Newton's method was used to solve the nonlinear problems and were resolved typically in 2 to 3 steps.

\begin{figure}
\centering
\textcolor{blue}{Reference} \hspace{2.1cm} EMAC \hspace{2.1cm} EMAC-Reg\\
\subfloat{
	\includegraphics[scale=.175]{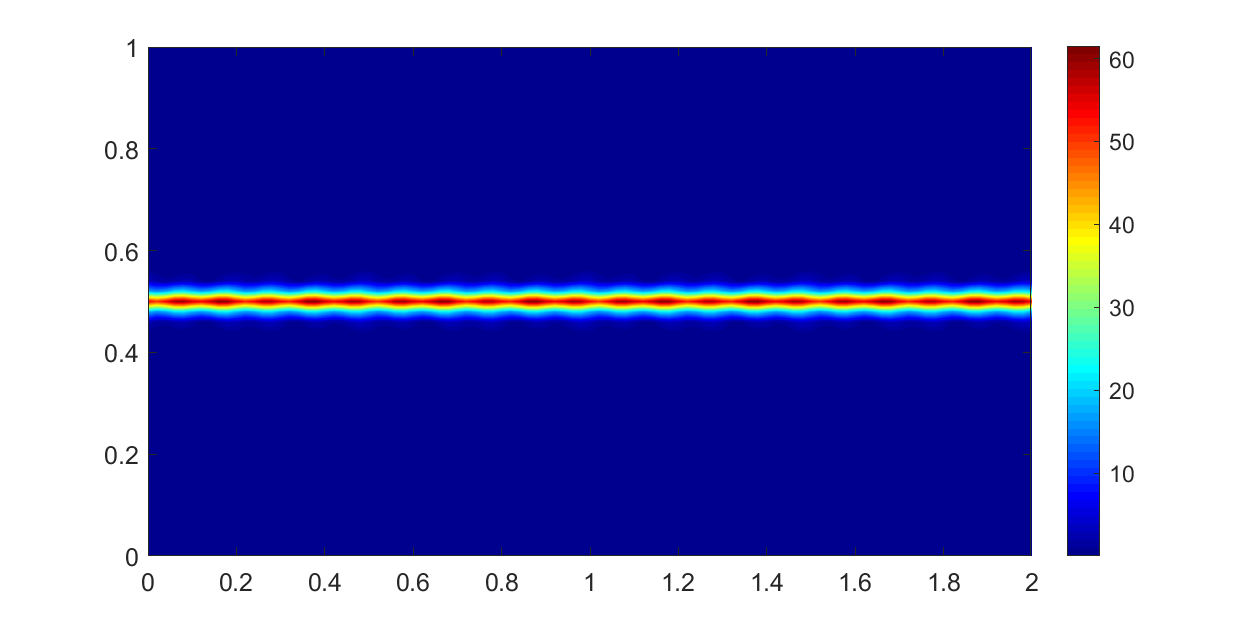}
}
\subfloat{
	\includegraphics[scale=.13]{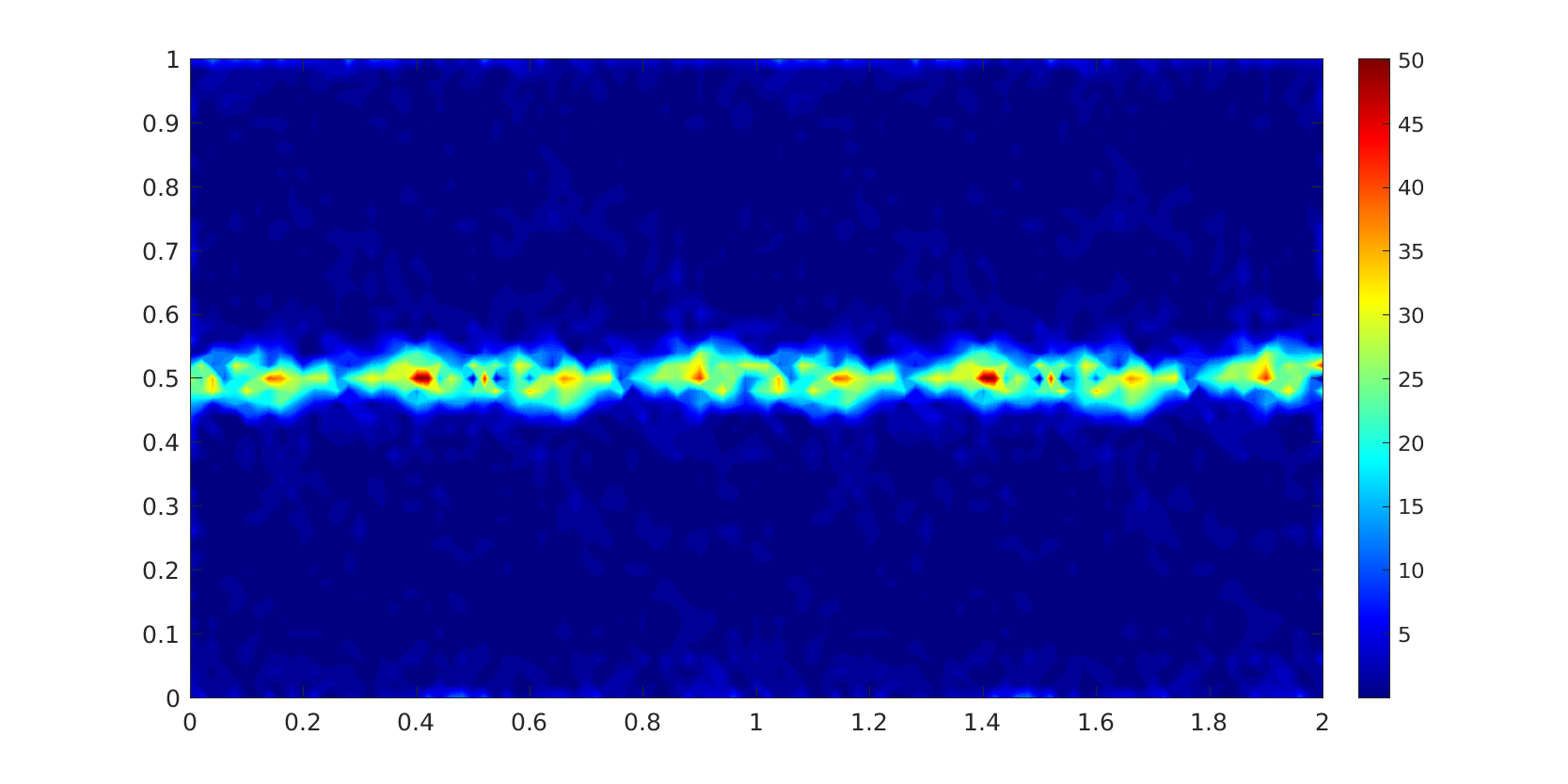}
}
\subfloat{
	\includegraphics[scale=.13]{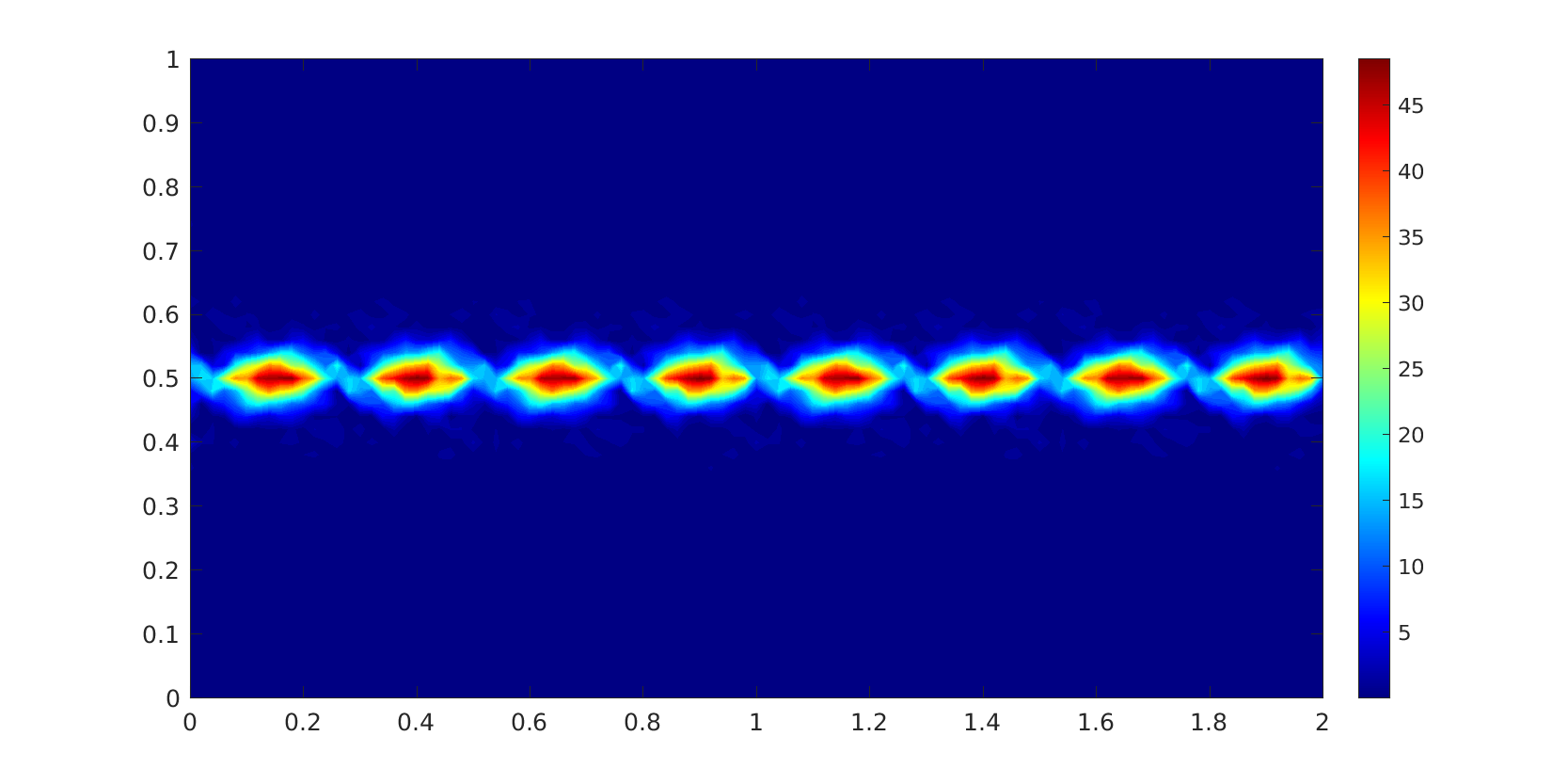}
}
\hspace{0mm}
\subfloat{
	\includegraphics[scale=.175]{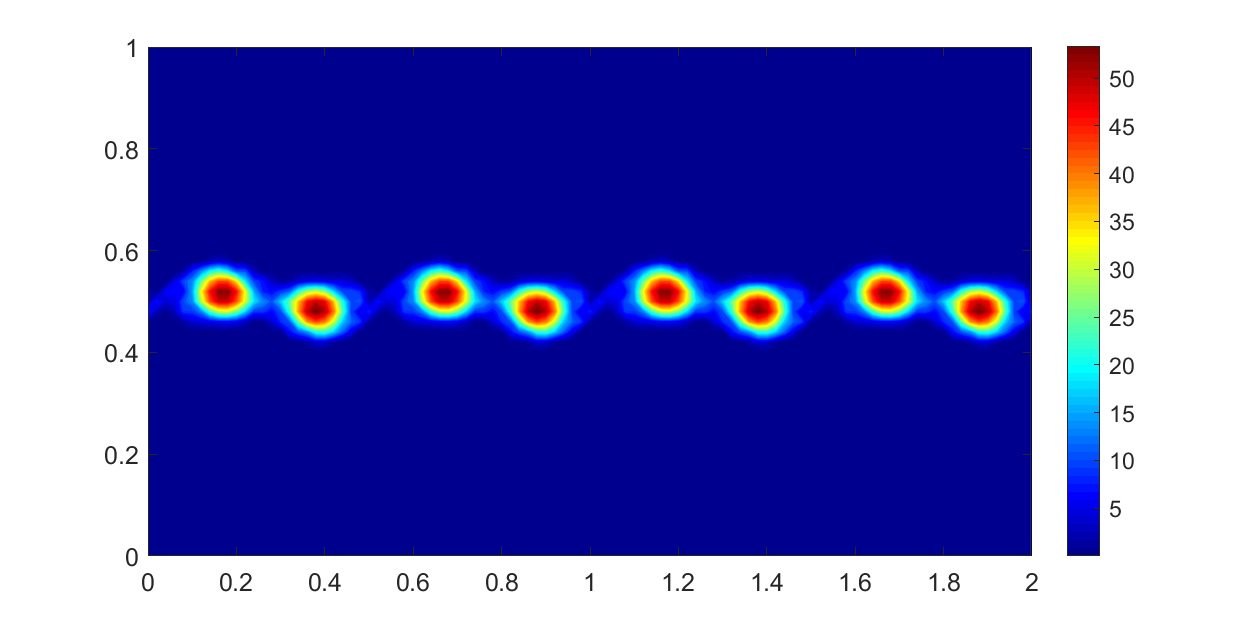}
}
\subfloat{
	\includegraphics[scale=.13]{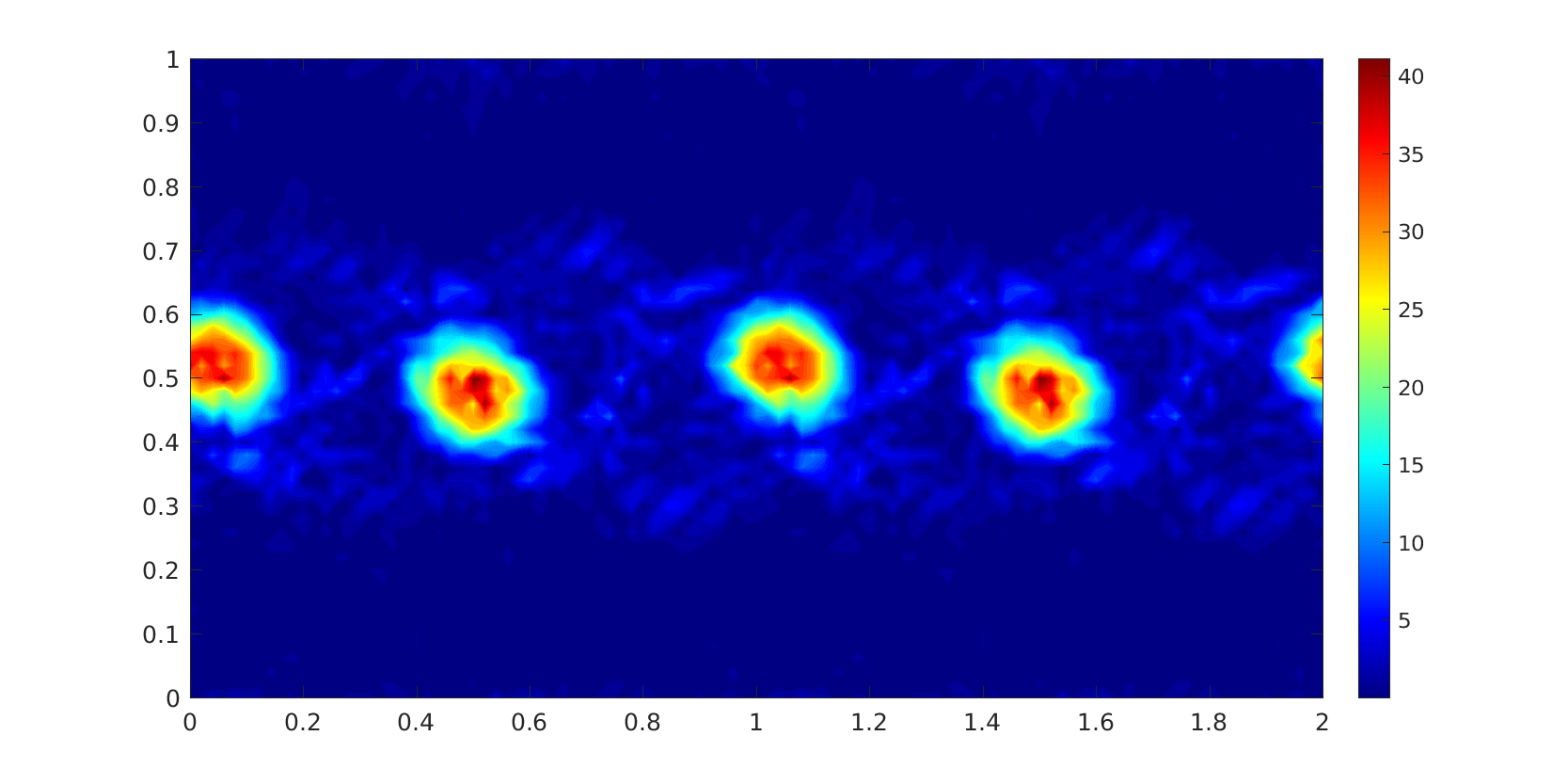}
}
\subfloat{
	\includegraphics[scale=.13]{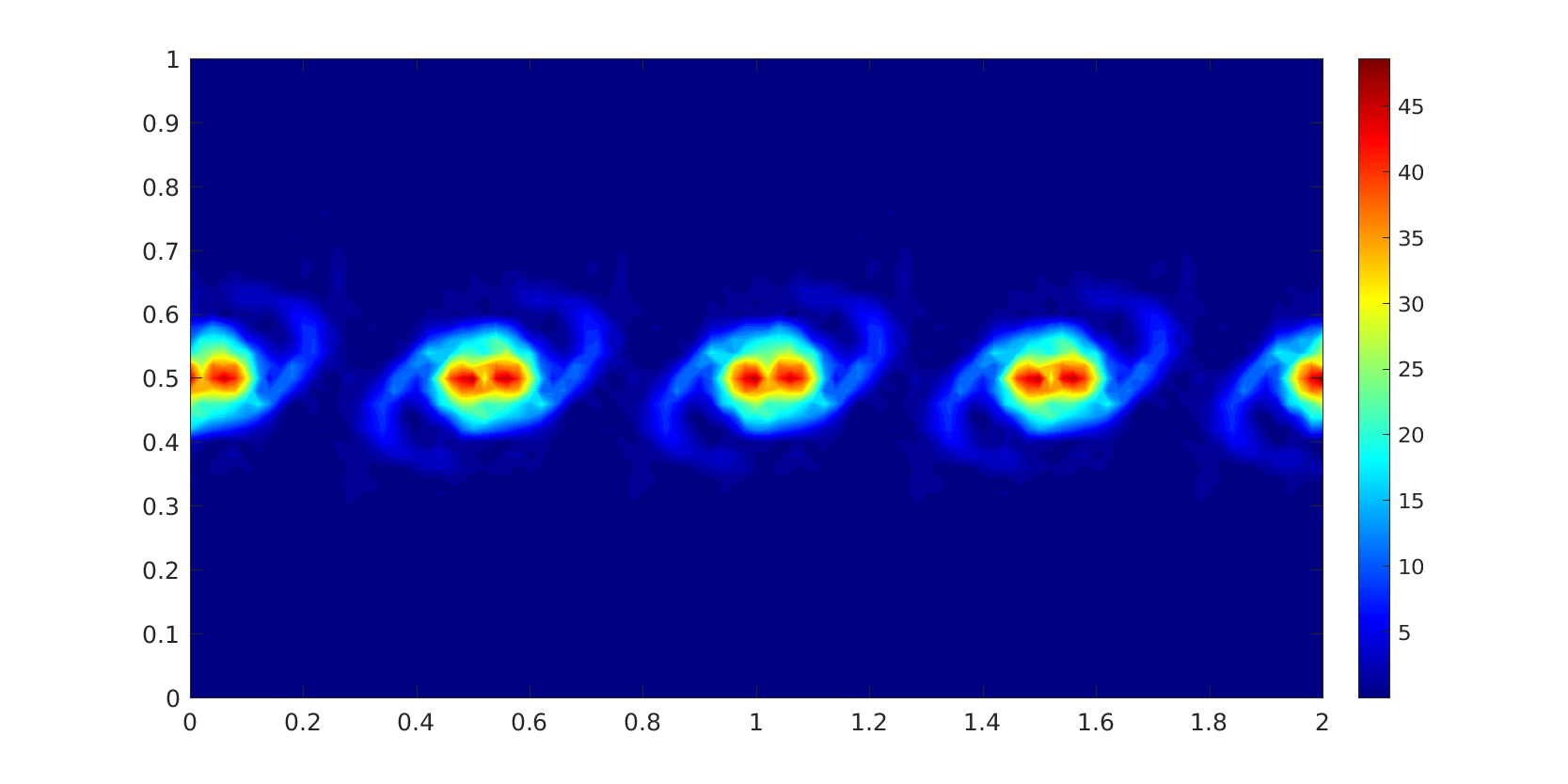}
}
\hspace{0mm}
\subfloat{
	\includegraphics[scale=.175]{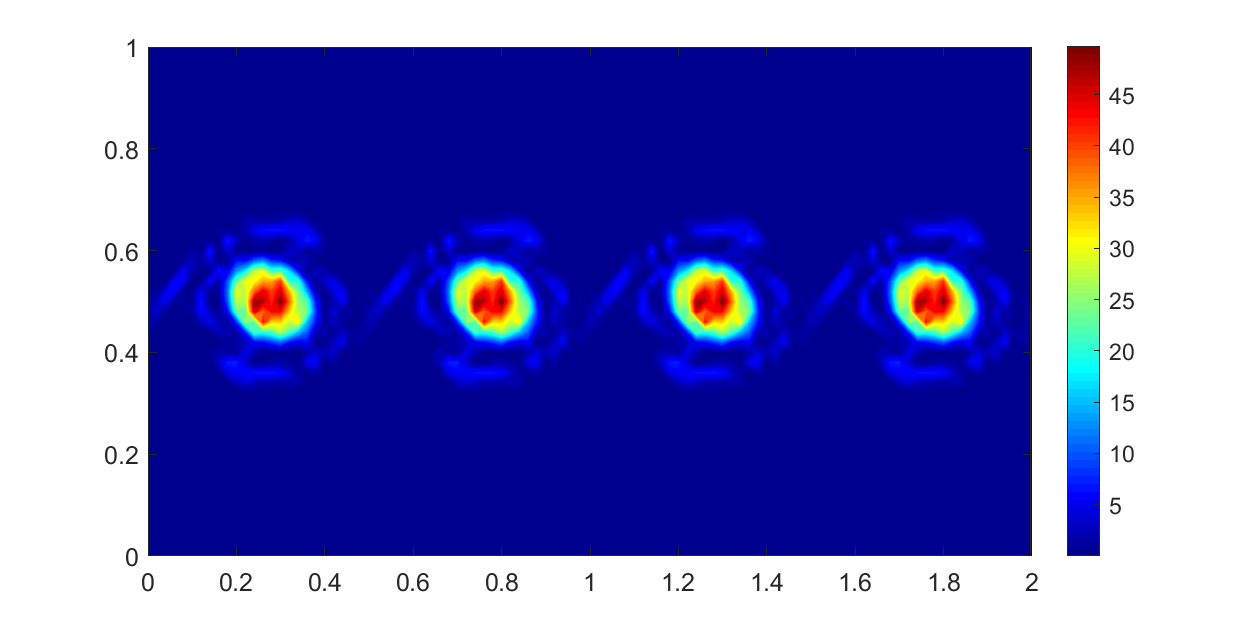}
}
\subfloat{
	\includegraphics[scale=.13]{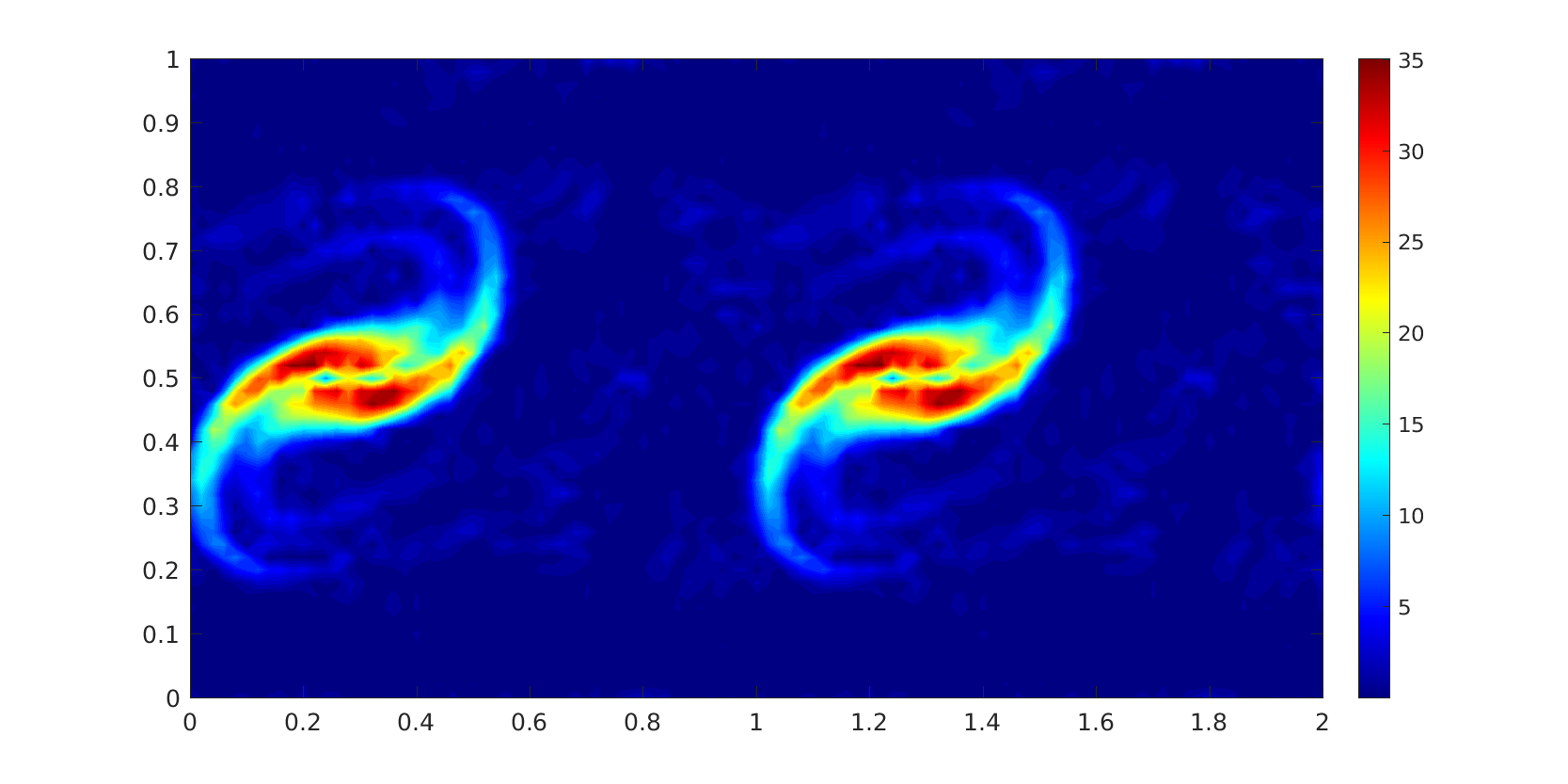}
}
\subfloat{
	\includegraphics[scale=.13]{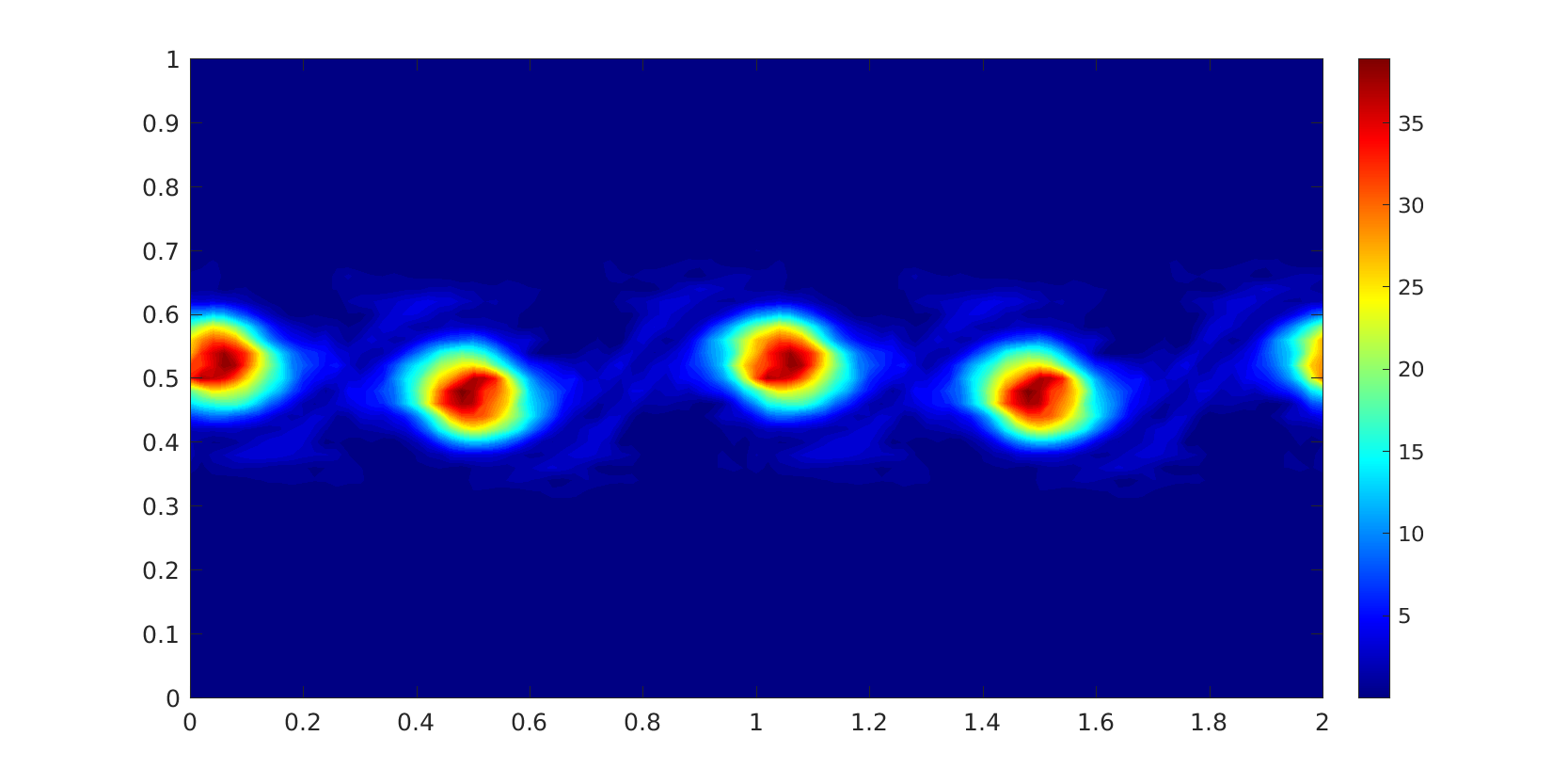}
}
\hspace{0mm}
\subfloat{
	\includegraphics[scale=.175]{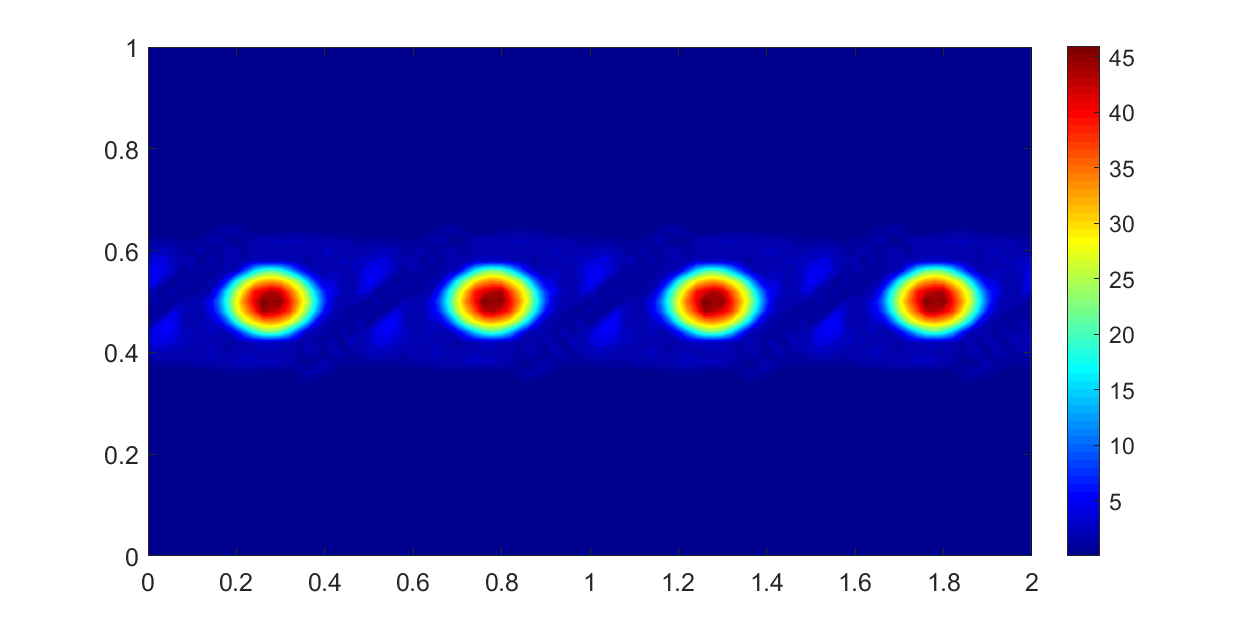}
}
\subfloat{
	\includegraphics[scale=.13]{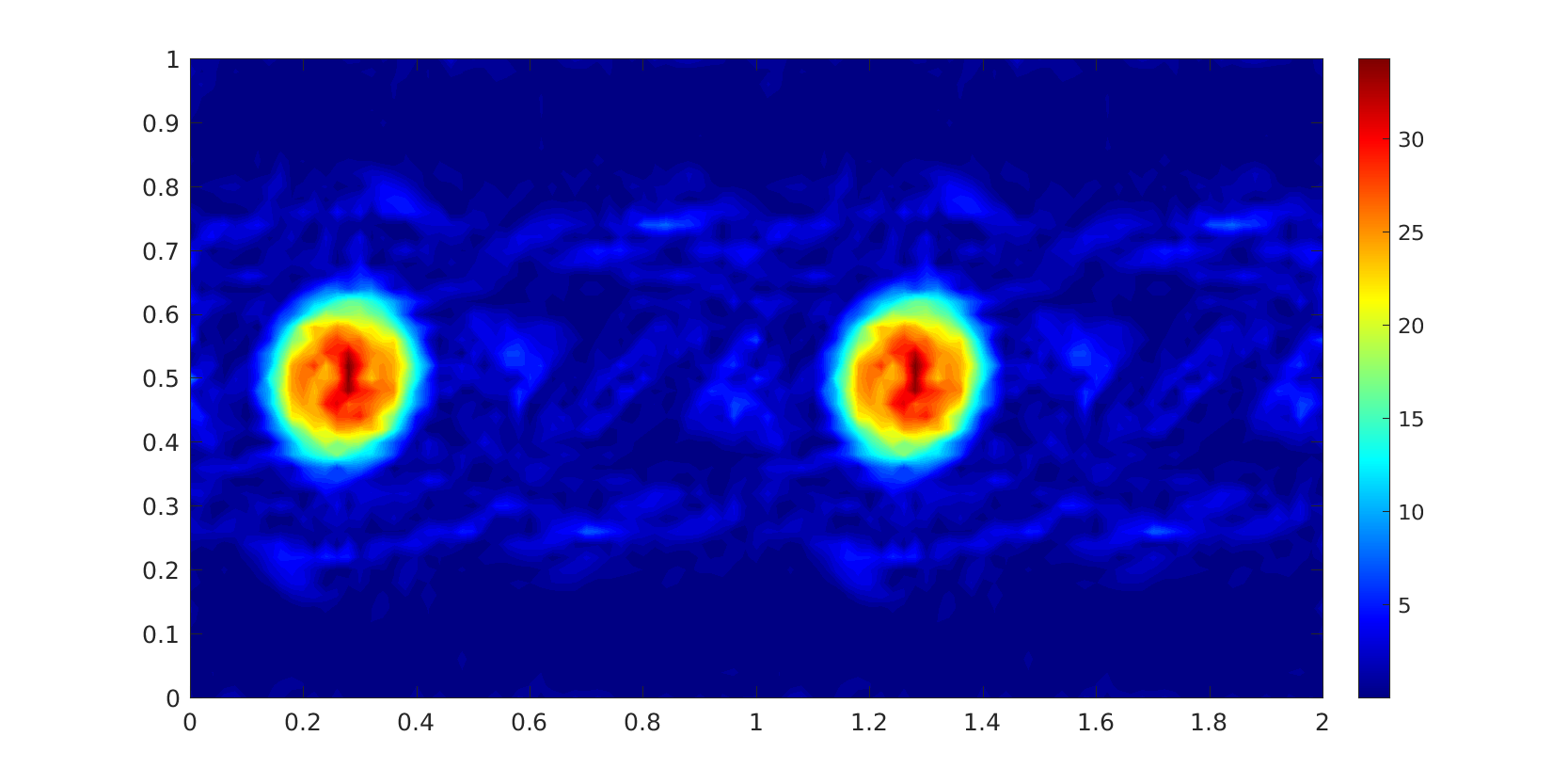}
}
\subfloat{
	\includegraphics[scale=.13]{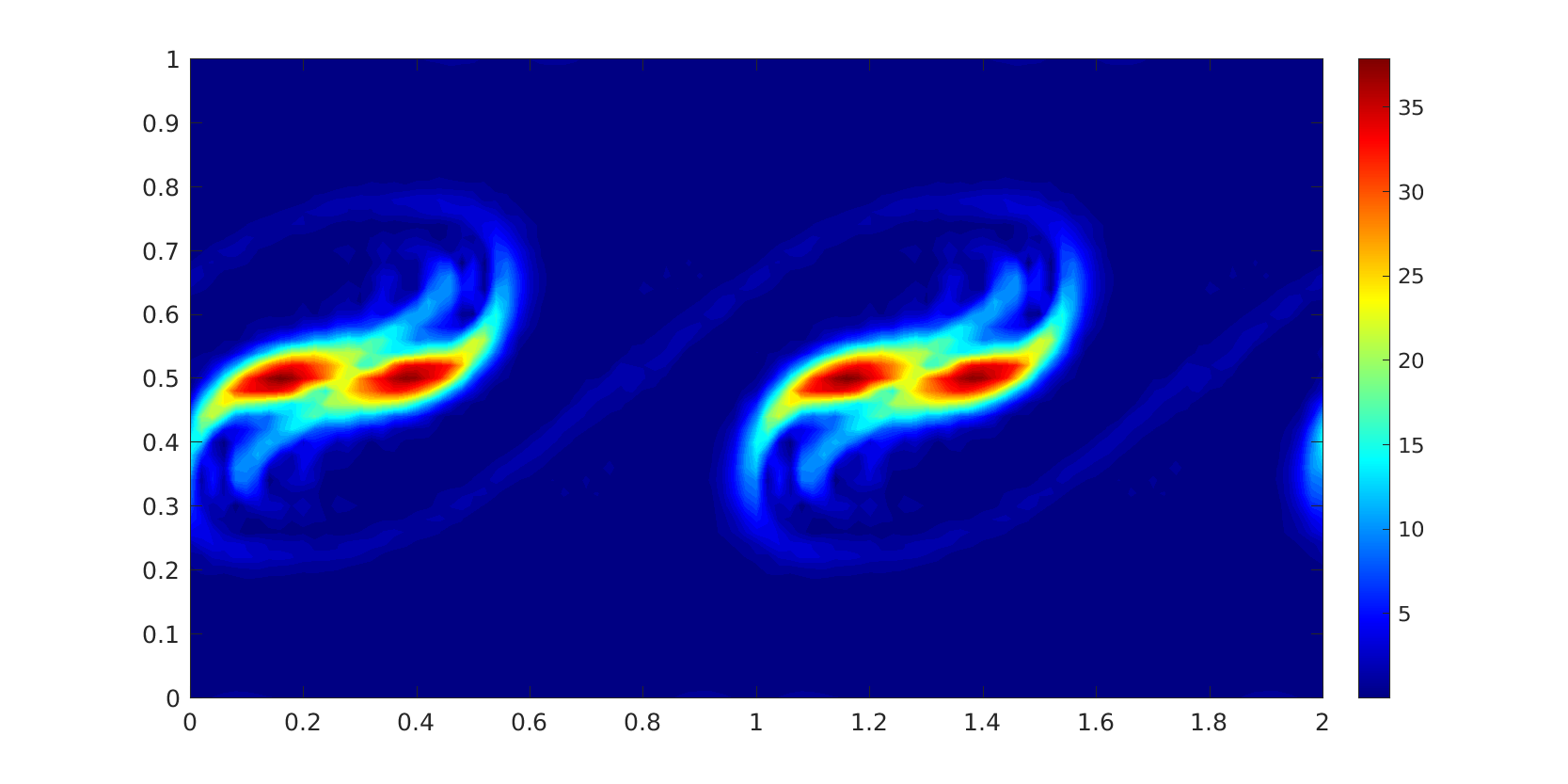}
}
\hspace{0mm}
\subfloat{
	\includegraphics[scale=.175]{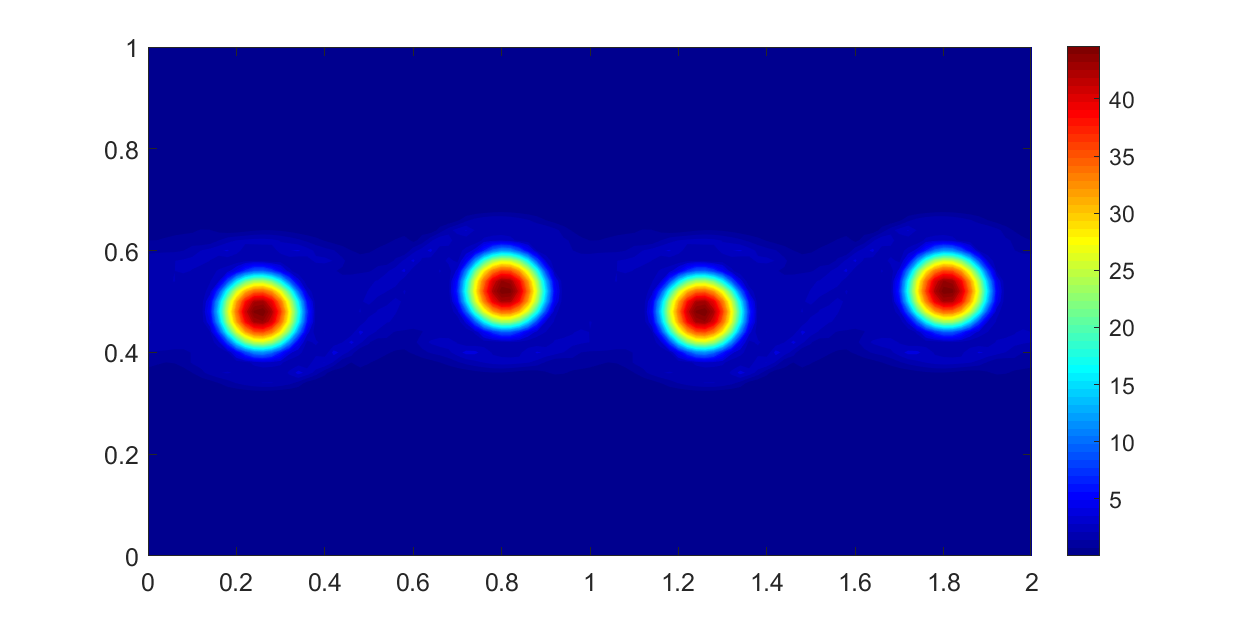}
}
\subfloat{
	\includegraphics[scale=.13]{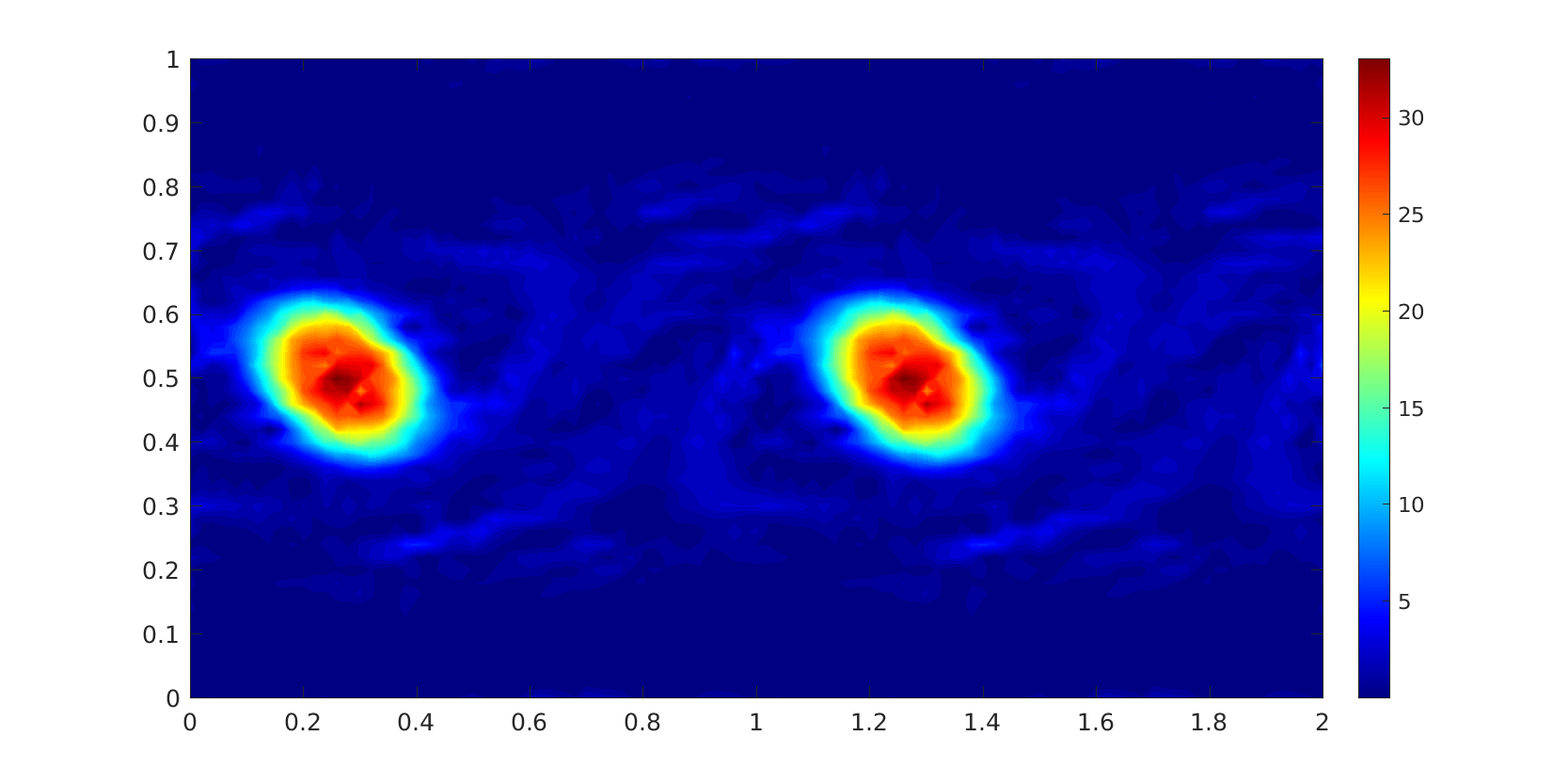}
}
\subfloat{
	\includegraphics[scale=.13]{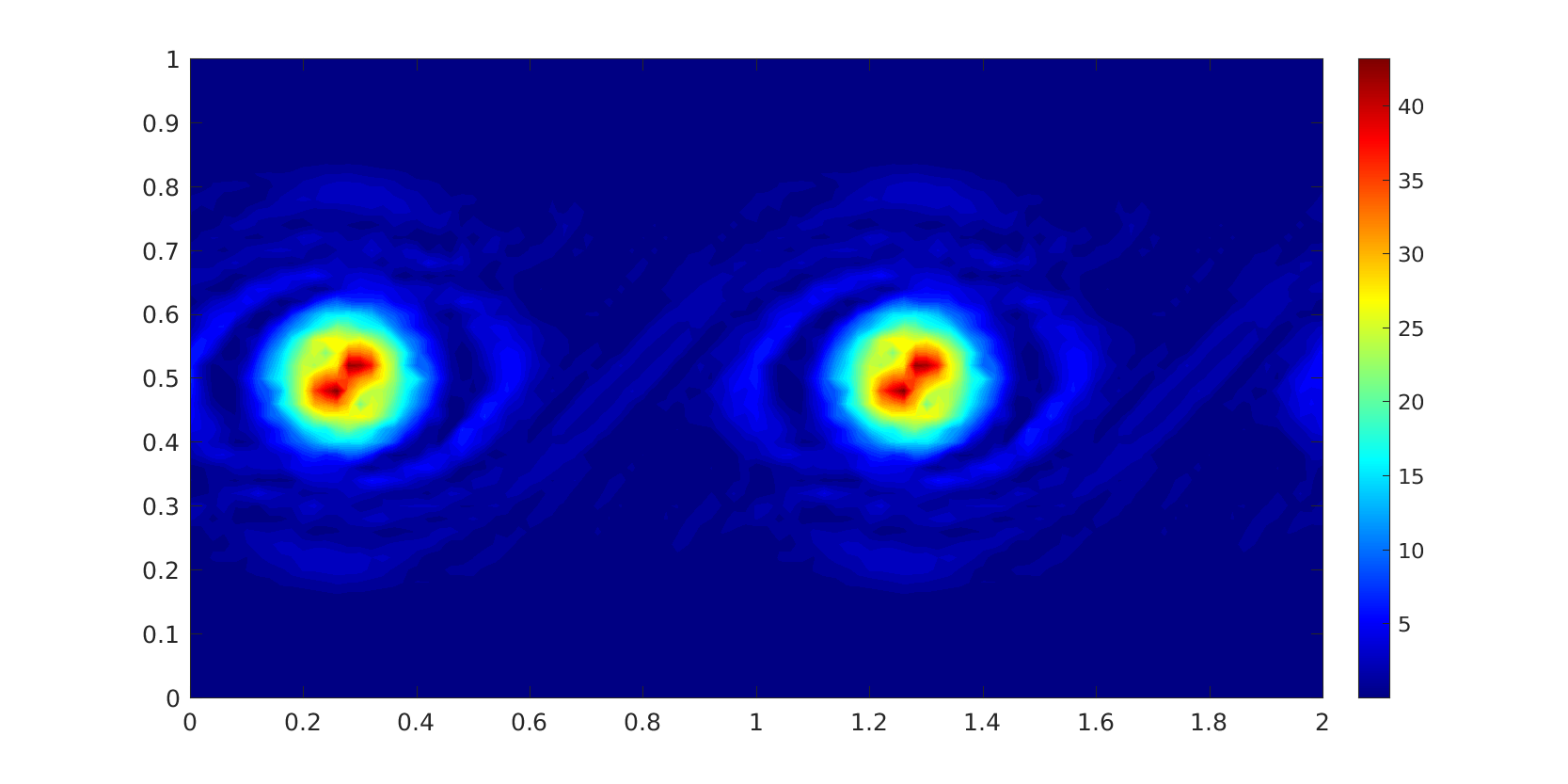}
}
\hspace{0mm}
\subfloat{
	\includegraphics[scale=.175]{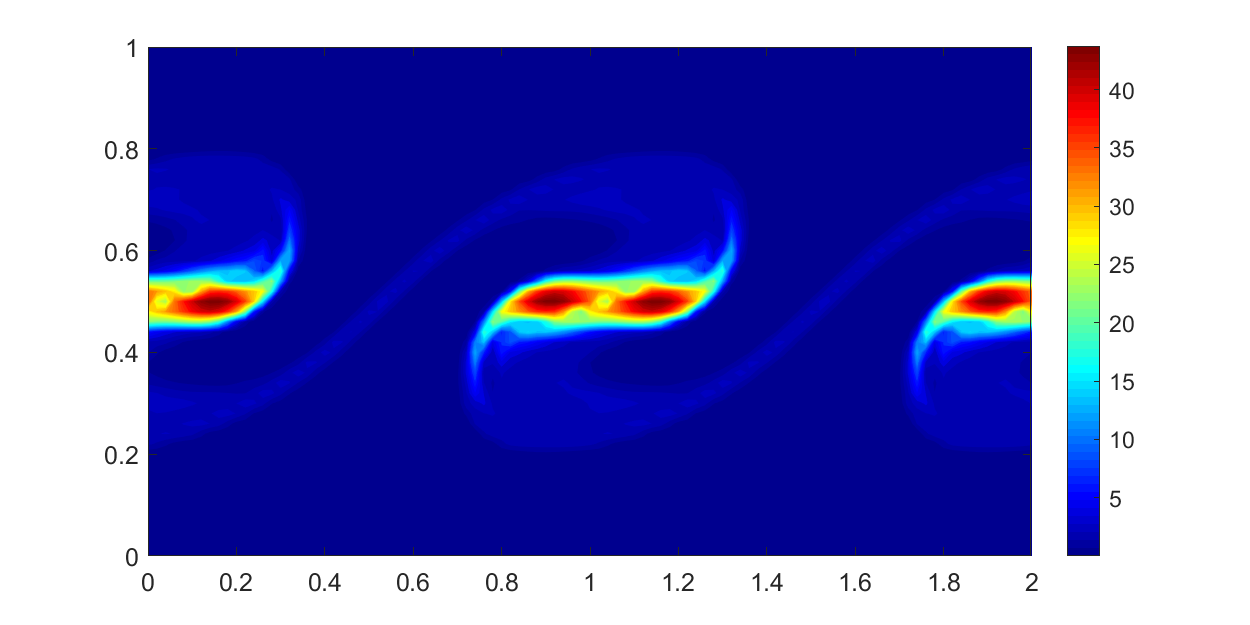}
}
\subfloat{
	\includegraphics[scale=.13]{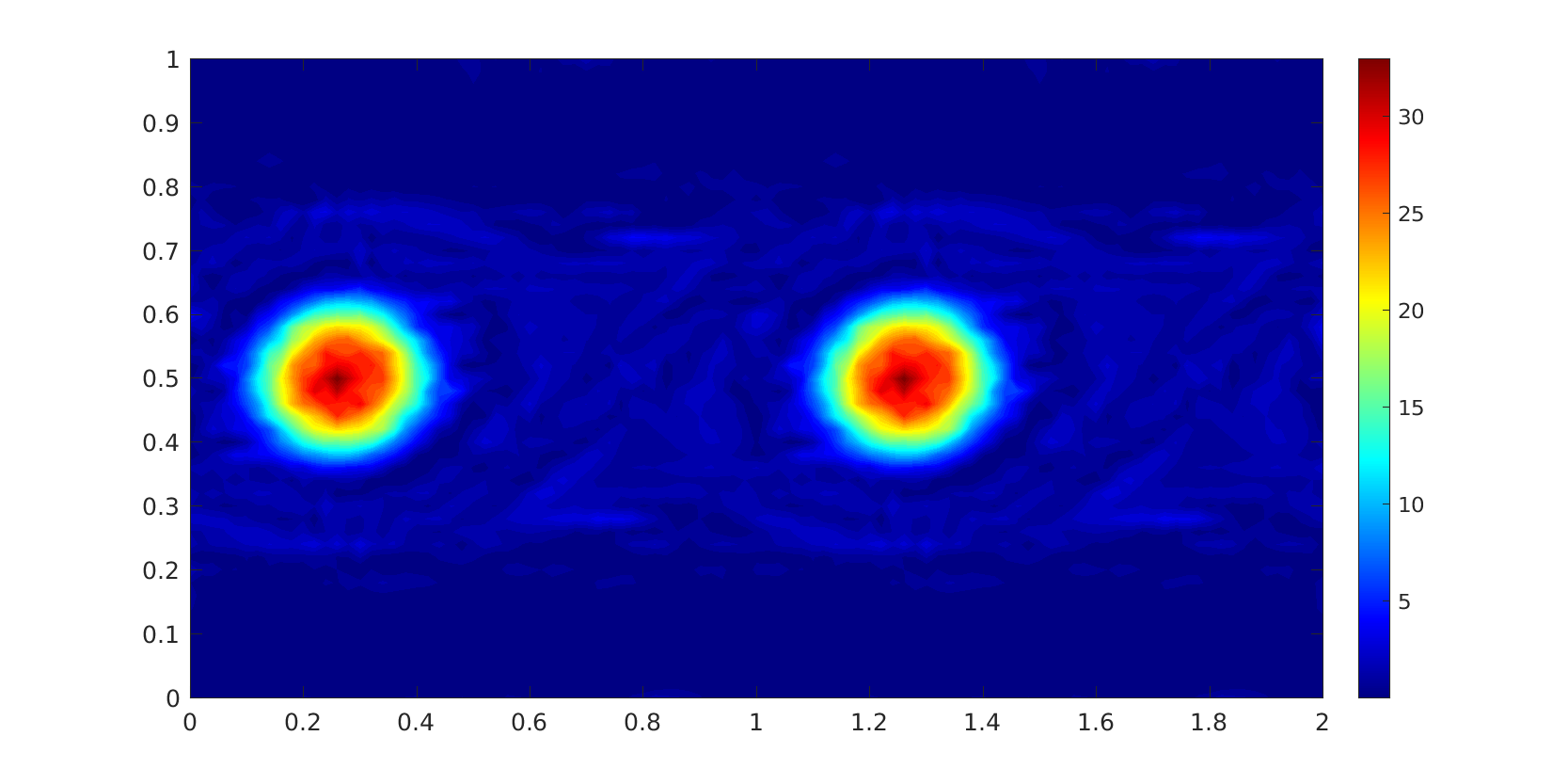}
}
\subfloat{
	\includegraphics[scale=.13]{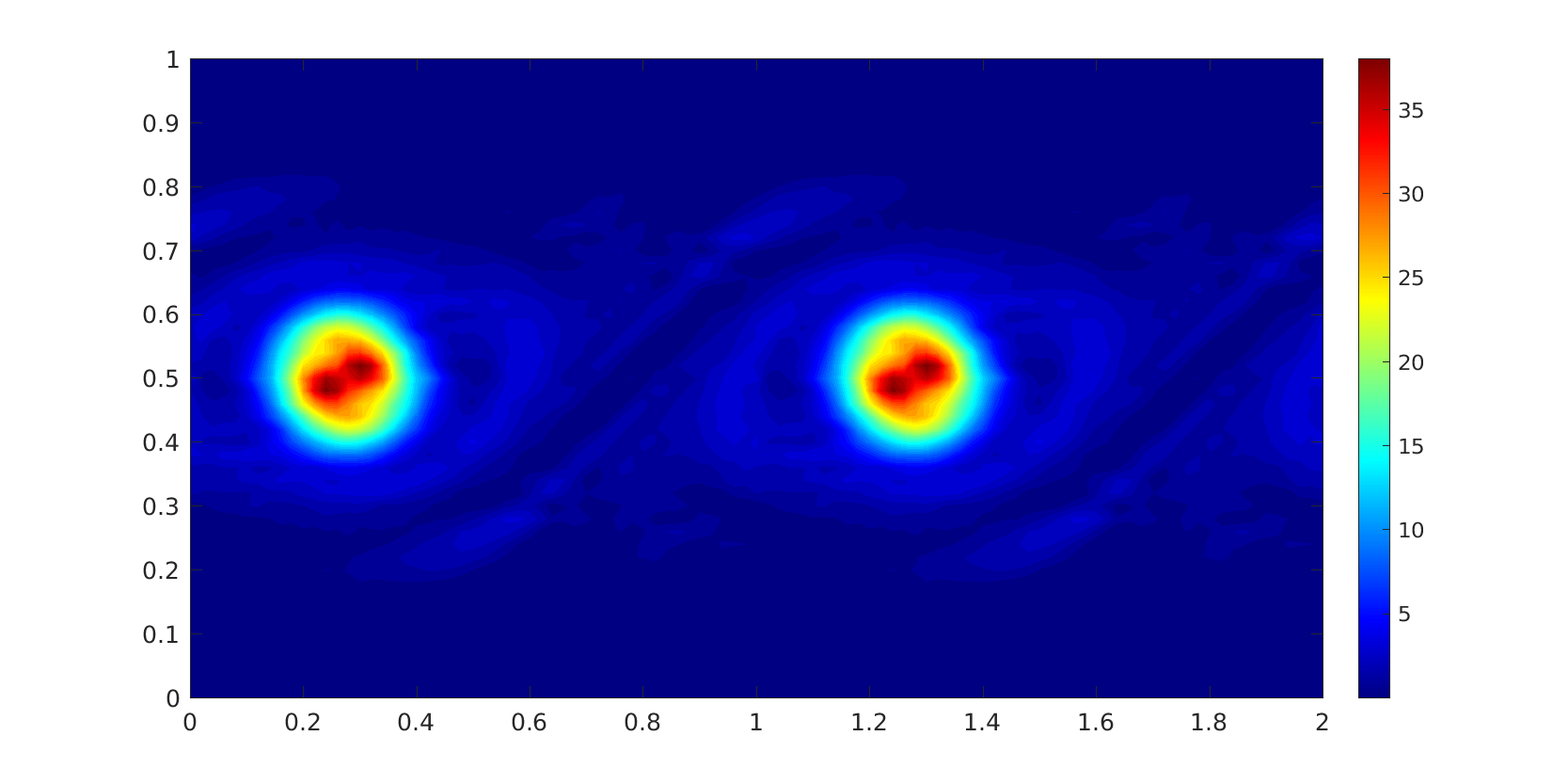}
}
\hspace{0mm}
\subfloat{
	\includegraphics[scale=.175]{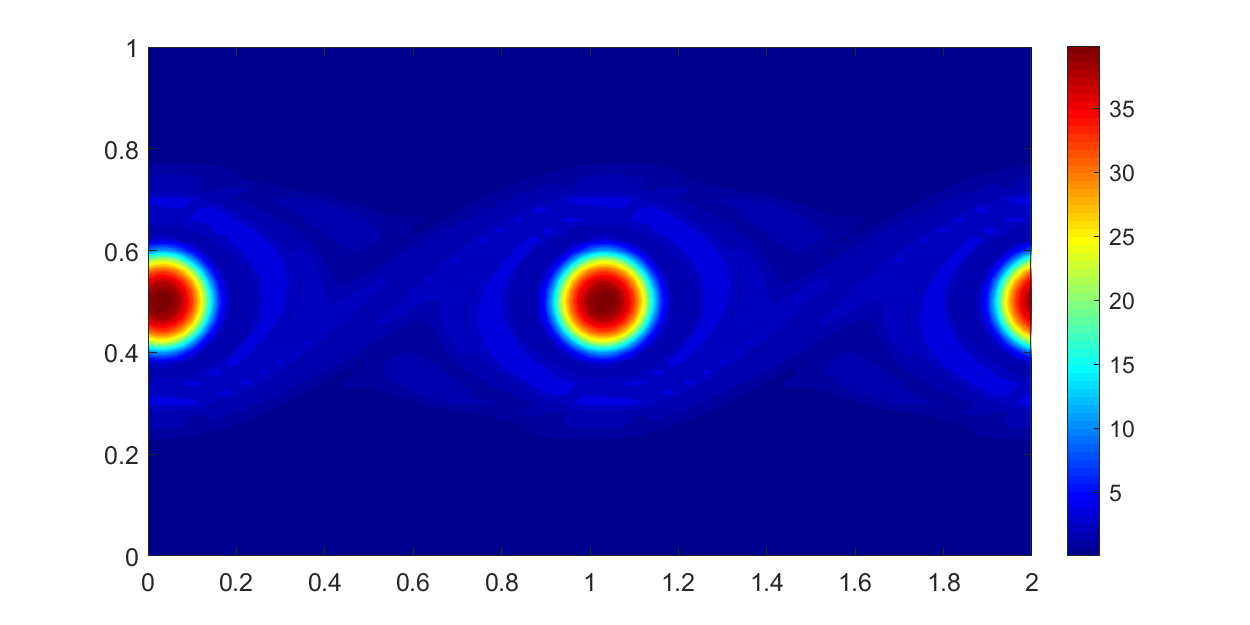}
}
\subfloat{
	\includegraphics[scale=.13]{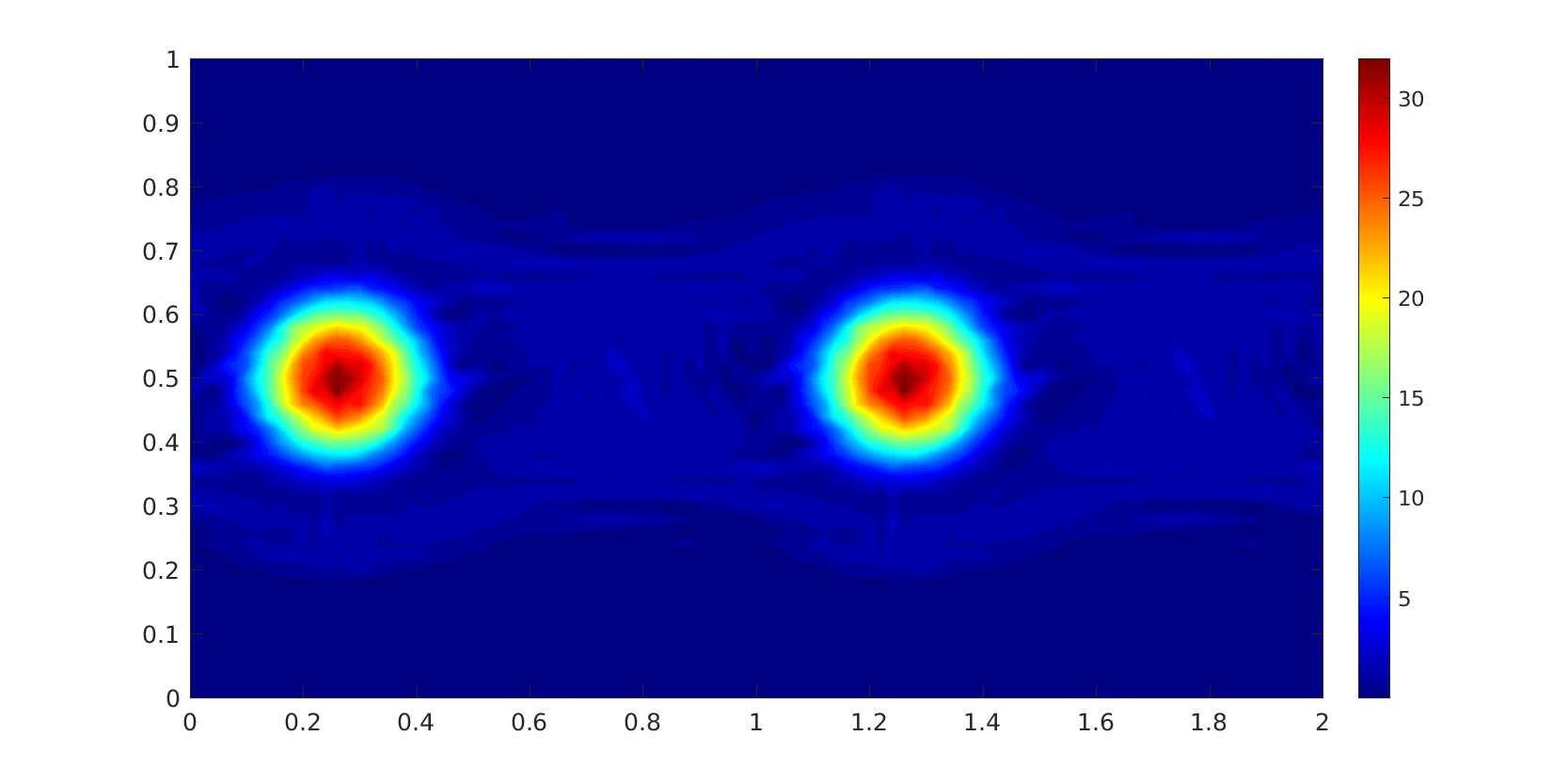}
}
\subfloat{
	\includegraphics[scale=.13]{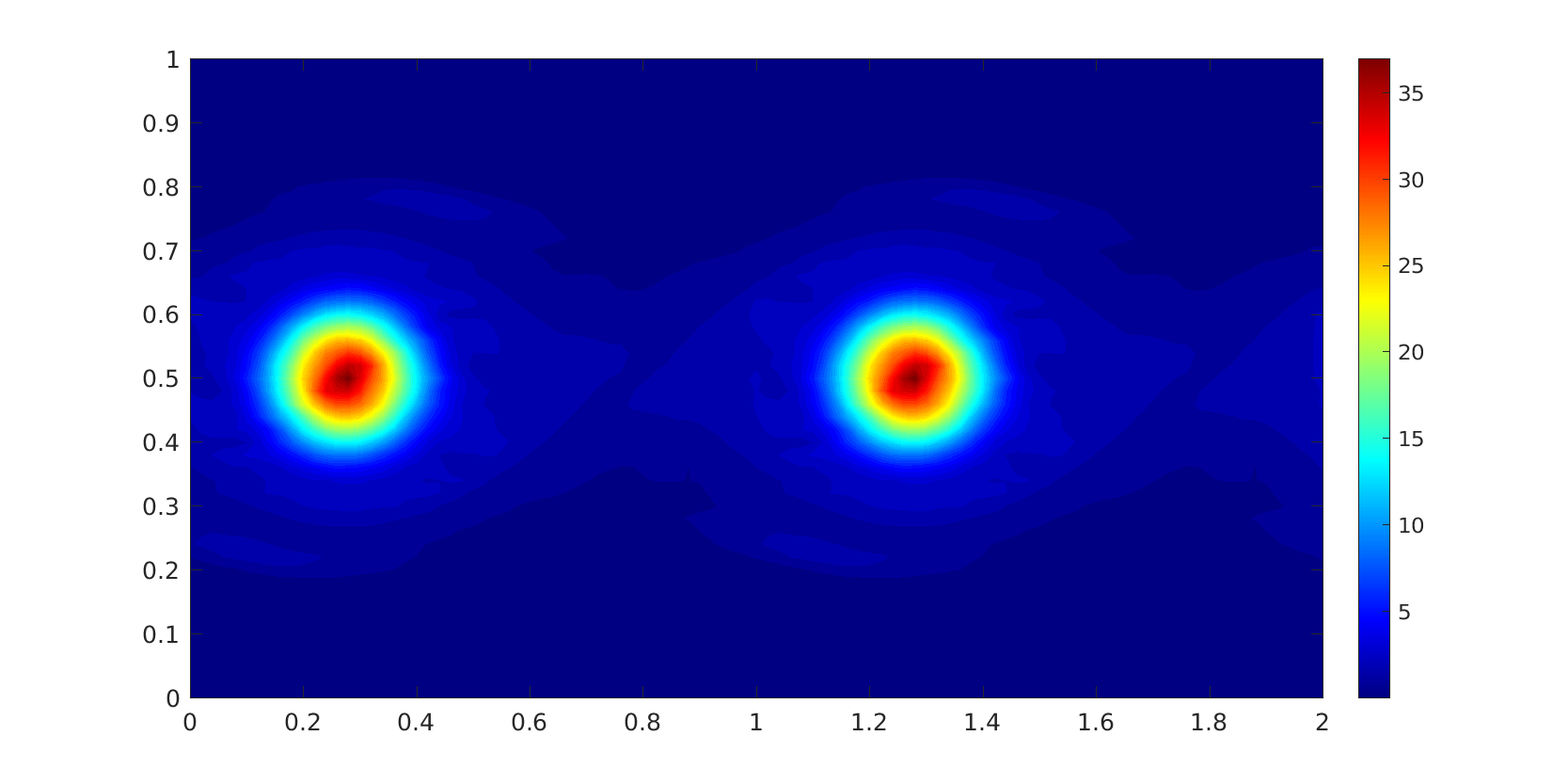}
}
\caption{\textcolor{blue}{Shown above are the vorticity contours for the reference solution (left), EMAC (middle), and EMAC-Reg (right) with mesh size 256 for the reference and 48 for EMAC and EMAC-Reg}, at times $t=$1, 2, 3, 4, 5, 6, 10, for $Re$=1000. \textcolor{blue}{Note the EMAC-Reg formulation used $\alpha=\frac{h}{3}$}.}
\label{fig:KH}
\end{figure}

Figure \ref{fig:KH} displays enstrophy contours at different times.  It was shown in \cite{OR20} that SKEW does not give resolved plots and show oscillations on a mesh significantly more refined than this one.  We can expect that on a coarser mesh that SKEW performs even more poorly, \textcolor{blue}{which was made clear during the experiment.}  

It was also shown in \cite{OR20} that EMAC performed well on a finer mesh, however on the coarse mesh it does not predict important physical phenomena that happens earlier on in the simulation.  Notice at time $t=1$, EMAC-Reg shows clear vortices, whereas EMAC is more grainy.  At times $t=3, \: 4$ we see a more refined plot from EMAC-Reg of the vortices' behavior than EMAC.

\textcolor{blue}{Comparing to the reference solution, the earlier times for EMAC-Reg closely resemble the resolved solution, much more than EMAC.}  As mentioned above, we observe several small vortices in the early phases of the experiment for EMAC-Reg, which is precisely how it is depicted in \cite{SJLLLS18}.  \textcolor{blue}{We also see that EMAC-Reg converges to a steady solution closer to the time where the reference converges.  We see convergence begin to form in the reference solution at $t=6$, whereas EMAC converges around $t=3$ and EMAC-Reg converges around $t=4$.}

\section{Conclusions}

We have introduced a new regularization model for fluid flow that not only \textcolor{blue}{conserves the same important physical quantities as EMAC using a properly chosen time integrator (Crank-Nicolson)}, but it also performs better than EMAC on coarser meshes.  We proved well-posedness of the semidiscrete EMAC-Reg scheme and provided an error analysis  that showed optimal rates and a Gronwall constant that does not depend on the Reynold's number.  We also showed numerically that EMAC-Reg outperforms related models on a coarser mesh in terms of $L^2$ error and accuracy over time.  Our results show the robustness of regularization models on coarser meshes along with models that conserve important physical quantities.

For future directions, further showing that EMAC-Reg conserves more quantities such as vorticity, enstrophy and helicity would make it more comparable to EMAC.  \textcolor{blue}{The theoretical analysis of the fully discrete scheme was not included in this paper, this is something that could be shown in a future work.}  We also only showed that energy, momentum, and angular momentum are only conserved given the proper time integrator.  It behooves us to show this for a general case in the future.  More numerical testing is also in order, in particular, testing EMAC-Reg in 3D.

\section*{Acknowledgments}

Clemson University is acknowledged for allowing a generous amount of computation time on the Palmetto cluster.

\bibliography{EMAC-Reg_bib_Rev_2}
\bibliographystyle{plain}

\begin{appendices}
\section{Momentum/angular momentum conservation of NS-$\alpha$ and Leray-$\alpha$ formulations}

Here we show that the NS-$\alpha$ and Leray-$\alpha$ formulations do not conserve momentum or angular momentum if $\div u\neq 0$ and $\div w\neq 0$ where $w$ represents the filtered velocity $\bar{u}$.  Further note $u=Fw$ where $F=-\alpha^2\Delta I+I$.

\subsection{NS-$\alpha$}
Recall the nonlinear term of the NS-$\alpha$ formulation is
\begin{align}
u_t+(\nabla \times u)\times w+\nabla p-\nu\Delta u=f.
\label{eq:nsalpha}
\end{align}
Test \eqref{eq:nsalpha} with $e_i$ for $i=1,2,3$.  After applying the space-time divergence theorem and rearranging some we get
\begin{align}
\lip{u_t}{e_i}+\lip{(\nabla \times u) \times w}{e_i}+\nu\lip{\nabla u}{\nabla e_i}=\lip{f}{e_i}.
\label{eq:nsalpha1}
\end{align}
Assuming $\nu=f=0$, \eqref{eq:nsalpha1} simplifies into
\begin{align*}
\frac{d}{dt}\lip{u}{e_i}+\lip{(\nabla \times u) \times w}{e_i}=0.
\end{align*}
If the nonlinear term is equal to zero, then we will have momentum conservation.  We now check this:
\begin{align*}
\lip{(\nabla \times u) \times w}{e_i}&=\lip{\nabla \times (w \times e_i)}{u}\\
&=\lip{(\nabla \cdot e_i)w}{u}-\lip{(\nabla \cdot w)e_i}{u}+\lip{e_i\cdot\nabla w}{u}-\lip{w\cdot \nabla e_i}{u},
\end{align*}
where the above two equalities come from vector identities.  Also note that because $e_i$ is a vector of scalars, $\lip{(\nabla \cdot e_i)w}{u}=\lip{w\cdot \nabla e_i}{u}=0$.  This leaves us with
\begin{align*}
\lip{(\nabla \times u) \times w}{e_i}=-\lip{(\nabla \cdot w)e_i}{u}+\lip{e_i\cdot\nabla w}{u},
\end{align*}
which we cannot conclude is zero, hence we cannot say that the NS-$\alpha$ formulation preserves momentum.

For angular momentum, we test \eqref{eq:nsalpha} with $\phi_i$ and the algebra works out similar to momentum,
\begin{align*}
\lip{(\nabla \times u) \times w}{\phi_i}=\lip{(\nabla \cdot \phi_i)w}{u}-\lip{(\nabla \cdot w)\phi_i}{u}+\lip{\phi_i\cdot\nabla w}{u}-\lip{w\cdot \nabla \phi_i}{u}.
\end{align*}
Since $\nabla \cdot \phi_i=0$ for $i=1,2,3$, we have 
\begin{align*}
\lip{(\nabla \cdot \phi_i)w}{u}=0.
\end{align*} 
Also recall using \eqref{eq:phi} in Theorem \ref{Conservation_Thm}, we have
\begin{align*}
\lip{w \cdot \nabla \phi_i}{u}=0.
\end{align*} 
This gives us
\begin{align*}
\lip{(\nabla \times u) \times w}{\phi_i}=-\lip{(\nabla \cdot w)\phi_i}{u}+\lip{\phi_i\cdot\nabla w}{u}.
\end{align*}
Much like with momentum, we cannot conclude that this quantity is zero, and we expect it is not zero.

\subsection{Leray-$\alpha$}

Recall the nonlinear term of the Leray-$\alpha$ formulation is
\begin{align}
u_t+w\cdot \nabla u+\nabla p-\nu\Delta u=f.
\label{eq:leray}
\end{align}
We test \eqref{eq:leray} with $e_i$ for $i=1,2,3$ and integrate.  Similar to \eqref{eq:nsalpha1}
\begin{align}
\lip{u_t}{e_i}+\lip{w\cdot \nabla u}{e_i}+\nu\lip{\nabla u}{\nabla e_i}=\lip{f}{e_i}.
\label{eq:leray1}
\end{align}
Assuming $\nu=f=0$, \eqref{eq:leray1} simplifies to
\begin{align*}
\frac{d}{dt}\lip{u}{e_i}+\lip{w\cdot \nabla u}{e_i}=0.
\end{align*}
If the nonlinear term is equal to zero, then we will have momentum conservation.  Using \eqref{eq:form1} on the nonlinear term we get
\begin{align*}
\lip{w\cdot \nabla u}{e_i}&=-\lip{w\cdot \nabla e_i}{u}-\lip{(\nabla \cdot w)u}{e_i}\\
&=-\lip{(\nabla \cdot w)u}{e_i},
\end{align*}
which is not zero when $\divergence w \neq 0$.  Hence momentum is not necessarily conserved.

For angular momentum we test \eqref{eq:leray} with $\phi_i$ for $i=1,2,3$ and it simplifies to
\begin{align*}
\frac{d}{dt}\lip{u}{\phi_i}+\lip{w\cdot \nabla u}{\phi_i}=0.
\end{align*}
Now similarly to the momentum proof, we have for the nonlinear term
\begin{align*}
\lip{w\cdot \nabla u}{\phi_i}&=-\lip{w\cdot \nabla \phi_i}{u}-\lip{(\nabla \cdot w)u}{\phi_i}\\
&=-\lip{(\nabla \cdot w)u}{\phi_i},
\end{align*}
where the first term disappears by applying \eqref{eq:phi} similarly to Theorem \ref{Conservation_Thm} (and the angular momentum proof for NS-$\alpha$ in Appendix A.1).  Thus the nonlinear term does not vanish, so angular momentum is not conserved.

\end{appendices}

\end{document}